\documentclass[preprint,12pt]{elsarticle}

\usepackage[b5paper, margin={0.5in,0.65in}]{geometry}
\usepackage{lmodern}% Latin Modern fonts
\usepackage[T1]{fontenc}% T1 font
\usepackage[%
  shortlabels,% snippt of labels
  inline,% predefined inline lists
]{enumitem}% Enumerate and itemize lists
\usepackage[%
  tocdep=2, %the depth of toc
]{secnum} % section number formulation
\usepackage{mathtools}% mathtools already load amsmath
\usepackage[%
  scr=stixtwofancy, % set mathscr to *
  cal=euler, % set mathcal to * stixtwoplain does not have bold
  bb=stixtwo, % set mathbb to *
]{mathalpha}% Read the manuel for font names
\usepackage{amssymb}% math symbols
\usepackage{amsthm}
\usepackage[Symbol]{upgreek}% Upper case of greek
\usepackage{mleftright}% mathleftright
\usepackage{interval} % Interval
\usepackage{xfrac} % Fractions
\usepackage{fixdif}% Differential operators
\usepackage{leftindex}% better left superscript
\usepackage{accents}% \accentset
\usepackage{aligned-overset}% 
\usepackage{mathcommand}% functionalities for defining math macros
\usepackage{tikz}% tikz
\usepackage[%
  pagebackref, %backref with page and hyper
  breaklinks, %allow line break in toc
]{hyperref}
\usepackage[title]{appendix}%
\usepackage[%
  hypcap=true,% anch at the beginning
  font=small,%
]{caption}% Caption
\usepackage[%
  capitalize,% capitalise cross-reference names,regardless of where they appear in the sentence
  nameinlink,% use name not the number as link
]{cleveref}
\hypersetup{%
  linktoc=all, % make text (section), page number (page), both (all) or nothing (none) be link on TOC, LOF and LOT
  colorlinks=true,
}
\AtBeginDocument{
\hypersetup{%
  urlcolor=cyan,
}}
\numberwithin{equation}{section}

\crefname{section}{\S}{\S}% Prefix for cited section
\NewDocumentCommand\Crefnameitem { m m m O{\textup} O{(\roman*)}} {%
  \Crefname{#1enumi}{#2}{#3} % crefname for items in #1
  \AtBeginEnvironment{#1}{%
    \crefalias{enumi}{#1enumi}%
    \setlist[enumerate,1]{
        label={#4{#5}.},
        ref={#5}
    }%
  }  
}

\NewDocumentMathCommand\odv{m}{\frac{\d}{\d{#1}}}
\NewDocumentMathCommand\pdv{m}{\frac{\partial}{\partial{#1}}}
\LoopCommands{GNZQRCFKTDWM}[#1]
  {\declaremathcommand#2{\mathbb#1}}

\LoopCommands{ADEFGHIJKMNOSUVWY}[c#1]
  {\declaremathcommand#2{\mathcal#1}}

\LoopCommands{abcduvw}[sp#1]
  {\declaremathcommand#2{\mathscr#1}}
  
\LoopCommands{FJLMNRU}[sf#1]
  {\declaremathcommand#2{\mathsf#1}}

\LoopCommands{mnpq}[#1#1]
  {\declaremathcommand#2{\mathfrak#1}}
\LoopCommands{hgxX}[#1]
  {\declaremathcommand#2{\mathfrak#1}}

\declaremathcommand\iu{\mathbb{i}}
\declaremathcommand\L{\mathcal{L}}
\declaremathcommand\U{\mathscr{U}}
\declaremathcommand\H{\mathsf{H}}
\newmathcommand\vac{\mathbb{1}}
\def\Res{{\rm Res}}
\def\wt{{\rm wt}}
\def\de{\delta}
\def\dim{{\rm dim}}

\def\End{{\rm End}}
\def\Aut{{\rm Aut}}
\def\reg{{\rm reg}}

\def\lf{\lfloor}
\def\rf{\rfloor}
\def\of{\overline}
\def \<{\left\langle}
\def \>{\right\rangle}
\def \dast{\underline{\ast}}
\def \uast{\overline{\ast}}
\newtheorem{theorem}{Theorem}[section]
\newtheorem{proposition}[theorem]{Proposition}
\newtheorem{corollary}[theorem]{Corollary}
\newtheorem{lemma}[theorem]{Lemma}

\newtheorem{definition}[theorem]{Definition}

\newtheorem{remark}[theorem]{Remark}

\allowdisplaybreaks
\begin{document}
\begin{frontmatter}



\title{Twisted regular representations and bimodules in vertex operator algebra theory}


\author{Yiyi Zhu}

\affiliation{organization={Department of mathematics, South China University of Technology},%Department and Organization
            addressline={381 Wushan Road}, 
            city={Guangzhou},
            postcode={510641}, 
            state={Guangdong},
            country={China}}
\begin{abstract}
In this paper, we use the twisted regular representation theory of vertex operator algebras to construct bimodules over twisted Zhu algebras, extending Haisheng Li's work in untwisted scenarios. Moreover, a conjecture of Dong and Jiang on bimodule theory is confirmed.
\end{abstract}


% \begin{highlights}
% \end{highlights}


\end{frontmatter}


\section{Introduction}
The (twisted) representation theory of a vertex operator algebra has been studied extensively in terms of the representation theory of various associative algebras associated with the vertex operator algebra (see \cite{KW94, Z96, DLM98a, DLM98b, DLM98c, M04}). Compared to the universal enveloping algebras associated with vertex operator algebras \cite{FZ92}, these associative algebras are smaller \cite{He17, Han20, Han22, HAN25} and easier to deal with, and they have been proved to be very useful in the study of representations of vertex operator algebras. Let $V$ be a vertex operator algebra, $g$ be an automorphism of $V$ with finite order $T$, a series of associative algebras $A_{g, n}(V)$ were defined for $n\in \frac{1}{T}$ in \cite{DLM98c}. These associative algebras $A_{g, n}(V)$ carry rich information about the representation of $V$. For example, there is a bijection between the set of inequivalent irreducible $g$-twisted $V$-modules and the set of inequivalent irreducible $A_{g, n}(V)$-modules, which is very useful in classifying irreducible $g$-twisted $V$-modules. Also, the semisimplicity of $A_{g, n}(V)$ for large enough $n$ implies the rationality of $V$.

The $A_{g, n}(V)$-$A_{g, m}(V)$-bimodule $A_{g, n, m}(V)$ theory for $m, n\in \frac{1}{T}\N$ developed by Dong and Jiang in \cite{DJ08a, DJ08b} greatly expanded the $A_{g, n}(V)$ theory with $A_{g, n, n}(V)=A_{g, n}(V)$ as associative algebras. $A_{g, n}(V)$ depicts all zero mode actions of $V$ on some subspace $\Omega_{n}(M)$ of an admissible $g$-twisted module $M$, while $A_{g, n, m}(V)$ depicts all $n-m$ mode actions. Because of this, Dong and Jiang were able to construct the generalized Verma admissible $g$-twisted $V$-module $\overline{M}(U)$ generated by an $A_{g, m}(V)$-module $U$ using these bimodules $A_{g, n, m}(V)$, which was first constructed in \cite{DLM98c} using a Lie algebraic approach. Roughly speaking, $\overline{M}(U)$ is the 'biggest' admissible $g$-twisted $V$-module generated by $U$. It plays an important role in the representation theory of $V$. From this point of view, these $A_{g, n, m}(V)$ help us understand the representation of $V$ better.

However, there is a setback of the $A_{g, n, m}(V)$ theory. It is almost impossible to compute these bimodules. Those various associative algebras mentioned above are all quotients of $V$, and there are explicit descriptions of the elements modulo out. For example, the Zhu algebra $A(V)=V/O(V)$ as vector spaces, where $O(V)$ is spanned by 
\begin{equation*}
    a\circ b=\Res_{x}x^{-2}Y\left((1+x)^{L_{(0)}}a, b\right)
\end{equation*}
for all $a, b\in V$. The Zhu algebras of almost all familiar vertex operator algebras are well known \cite{FZ92}. But for $A_{g, n, m}(V)$, they are more complicated than Zhu algebras, the subspace $O_{g, n, m}(V)$ quotient out by Dong and Jiang in \cite{DJ08a, DJ08b} consists of three parts,
\begin{equation*}
    O_{g, n, m}(V)=O'_{g, n, m}(V)+O''_{g, n, m}(V)+O'''_{g, n, m}(V),
\end{equation*}
where $O'_{g, n, m}(V)$ is a generalization of $O(V)$ in Zhu algebra case, while $O''_{g, n, m}(V)$ and $O'''_{g, n, m}(V)$ are new. From the representation point of view, it is natural to quotient $O''_{g, n, m}(V)$ and $O''_{g, n, m}(V)$ out, but it is hard to write them down explicitly. Dong and Jiang conjectured that $O_{g, n, m}(V)=O'_{g, n, m}(V)$. A recent result in \cite{HAN25} shows that $O'''_{g, n, m}(V)$ is redundant. In this paper, we confirm Dong and Jiang's conjecture by making use of the regular representation theory.

The regular representation theory of vertex operator algebras was developed by Li in \cite{Li02}. Let $W$ be a weak $V$-module. For any complex number $z\in \C$, there is a weak $V\otimes V$-module structure on a subspace $\mathfrak{D}^{(z)}(W)$ inside the full dual $W^{*}$ of $W$. It turns out to be closely related to various Zhu algebras and their bimodules. The constructions of various Zhu algebras were recovered by Li using the regular presentation theory, see \cite{Li01, Li01a}. Moreover, he gave a more conceptual proof of Frenkel-Zhu's fusion rules theorem in \cite{Li01}. The most important inspiration of Li's regular representation theory to this paper is that he proved that $V/O'_{id, n, m}(V)$ has an $A_{id, n}(V)$-$A_{id, m}(V)$-bimodule structure and recovered many Dong and Jiang's results in \cite{DJ08a}, which almost confirmed Dong and Jiang's conjecture in untwisted scenario. 

This paper aims to generalize Li's work in \cite{Li22} to the twisted scenario using the twisted regular representation theory of vertex operator algebras developed in \cite{LS23} and confirm Dong and Jiang's conjecture in \cite{DJ08a, DJ08b}.

Let $g_1, g_2$ be two finite-order mutually-commuting automorphisms of $V$ such that $g_1^T=g_2^T=1$, $W$ be a weak $(g_1g_2)^{-1}$-twisted $V$-module, and $m, n\in \frac{1}{T}\N$. We follow Li's idea in \cite{Li22} to study two quotient spaces of $W$, $A^{\dag}_{g_2^{-1}, g_1, n, m}(W)$ and $A^{\diamond}_{g_2^{-1}, g_1, n, m}(W)$. Since $g_1$ and $g_2$ commute with each other, $V$ has the following eigenspace decomposition:
\begin{equation*}
    V=\oplus_{0\leq j_1, j_2<T}V^{(j_1, j_2)},
\end{equation*}
where 
\begin{equation*}
    V^{(j_1, j_2)}=\{v\in V\mid g_kv=e^{\frac{2\pi \iu j_k}{T}}v, k=1, 2\}.
\end{equation*}
For homogeneous $a\in V^{(j_1, j_2)}$, and $v\in W$, define
\begin{equation*}
    a\circ_{g_2^{-1}, g_1, m}^{n} v=\Res_{x}\frac{(1+x)^{\lambda(m, j_1)}}{x^{\lambda(m, j_1)+\lambda(n, j_2)+2}}Y_{W}\left((1+x)^{L_{(0)}}a, x\right)v,
\end{equation*}
where $\lambda(x, r)$ is defined for any number $x\in \frac{1}{T}\N$ and integer $0\leq r<T$ by \labelcref{eq:lambda-number}.
Let $O^{\dag}_{g_2^{-1}, g_1, n, m}(W)$ be the space spanned by all $a\circ_{g_2^{-1}, g_1, m}^{n} v$ for $a\in V$ and $v\in W$, and let 
\begin{equation*}
    O'_{g_2^{-1}, g_1, n, m}(W)=O^{\dag}_{g_2^{-1}, g_1, n, m}(W)+(L_{(-1)}+L_{(0)}+m-n)(W).
\end{equation*}
Set 
\begin{equation*}
    A^{\dag}_{g_2^{-1}, g_1, n, m}(W)=W/O^{\dag}_{g_2^{-1}, g_1, n, m}(W)
\end{equation*}
and 
\begin{equation*}
    A^{\diamond}_{g_2^{-1}, g_1, n, m}(W)=W/O^{'}_{g_2^{-1}, g_1, n, m}(W).
\end{equation*} 

We show that both $A^{\dag}_{g_2^{-1}, g_1, n, m}(W)$ and $A^{\diamond}_{g_2^{-1}, g_1, n, m}(W)$ are $A_{g_2^{-1}, n}(V)$-$A_{g_1, m}(V)$-bimodules. For an $A_{g_1, m}(V)$-module $U$, we show 
\begin{equation*}
    A^{\diamond}_{g_2^{-1}, g_1, \square, m}(W)\otimes_{A_{g_1, m}(V)}U=\oplus_{n\in \frac{1}{T}\N} A^{\diamond}_{g_2^{-1}, g_1, n, m}(W)\otimes_{A_{g_1, m}(V)}U
\end{equation*}
is an admissible $g_2^{-1}$-twisted module.

In the special case $W=V, g_2^{-1}=g_1$, we have $O'_{g_1, g_1, n, m}(V)=O'_{g_1, n, m}(V)$ by definition. Hence $A_{g_1, n, m}(V)$ defined by Dong and Jiang is a quotient of both $A^{\dag}_{g_1, g_1, n, m}(V)$ and $A^{\diamond}_{g_1, g_1, n, m}(V)$. The left actions of these two bimodules are the same as the left action defined by Dong and Jiang in \cite{DJ08b} by definition. Though it looks different, we prove that the right action on $A^{\diamond}_{g_1, g_1, n, m}(V)$ is also the same as the one in \cite{DJ08b}. Most importantly, with the help of some results of Dong and Jiang, we prove that $O''_{g_1, n, m}(V)+O'''_{g_1, n, m}(V)\subseteq O'_{g_1, n, m}(V)$, which confirms the Dong and Jiang's conjecture.

This paper is organized as follows: In Section 2, we review some basics of the twisted representation theory of vertex operator algebras; In Section 3, we study the space $A^{\dag}_{g_2^{-1}, g_1, n, m}(W)$ and show that it has an $A_{g_2^{-1}, n}(V)$-$A_{g_1, m}(V)$-bimodule structure; In Section 4, we study the space $A^{\diamond}_{g_2^{-1}, g_1, n, m}(W)$ and show that it also has an $A_{g_2^{-1}, n}(V)$-$A_{g_1, m}(V)$-bimodule structure and recover some results of Dong and Jiang; In Section 5, we confirm Dong and Jiang's conjecture in \cite{DJ08a, DJ08b}.

\section{Preliminaries}
We refer to \cite{B86, FHL93, LL04} for the vertex operator algebra theory basics. Throughout this paper, for $\lambda\in \C$,  we fix
\begin{equation}
    (-1)^{\lambda}=e^{\pi \lambda \iu}.
\end{equation}

Let $(V, Y, \vac, \omega)$ be a vertex operator algebra, $g$ a finite-order automorphism of $V$, and $T$ a fixed positive integer such that $g^{T} = 1$. Then \[V=\oplus_{r=0}^{T-1} V^r, \text{\quad where } V^r=\{a\in V\mid ga=e^{\frac{2\pi \iu r}{T}}a\}.\]
\begin{definition}\cite{DLM98a}\label{def:g-twisted-module}
    A weak $g$-twisted $V$-module $M$ is a vector space equipped with a linear map $Y_M(-, x)$ from $V$ to $(\End$M$)[[x^{-\frac{1}{T}}, x^{\frac{1}{T}}]]$ such that: 
    \begin{itemize}
        \item For $a\in V^r, v\in M$, $Y_M(a, x)v=\sum_{n\in \frac{1}{T}\Z}a_{(n)}x^{-n-1} \in x^{-\frac{r}{T}}M((x))$;
        \item $Y_M(\vac, x)=id_M$;
        \item \textbf{Jacobi identity}: For $a\in V^r, b \in V$,
        \[x_0^{-1}\delta(\frac{x_1-x_2}{x_0})Y_M(a, x_1)Y_M(b, x_2)-x_0^{-1}\delta(\frac{-x_2+x_1}{x_0})Y_M(b, x_2)Y_M(a, x_1)\]
        \[=x_1^{-1}\delta(\frac{x_2+x_0}{x_1})(\frac{x_2+x_0}{x_1})^{\frac{r}{T}}Y_{M}(Y(a, x_0)b, x_2).\]
    \end{itemize}
\end{definition}
\begin{definition}\label{def:admissible-module}
    A weak $g$-twisted $V$-module $M$ is called admissible if there exists a $\frac{1}{T}\mathbb{N}$-grading $M = \oplus_{n\in \frac{1}{T}\mathbb{N}}M(n)$ such that for homogeneous $a\in V$, $a_{(n)}M(m)\subseteq M(m+\wt a-n-1)$. 
\end{definition} 
\begin{definition}\label{def:ordinary-module}
    A $g$-twisted $V$-module $M$ is a weak $g$-twisted $V$-module such that the following hold:
    \begin{itemize}
        \item $M=\oplus_{\lambda\in \C} M_{\lambda}$ with $\dim M_{\lambda}<\infty$ for any $\lambda\in \C$, and for any $\lambda_{0}\in \C$,  $M_{\lambda_{0}+\frac{n}{T}}=0$ for all sufficiently negative integers $n$;
        \item $M_{\lambda}=\{v\in M \mid L_{(0)}v=\lambda v\}$, where $L_{(0)}$ is the component operator in $Y_{M}(\omega, x)=\sum_{n\in \Z}L_{(n)}x^{-n-2}$.
    \end{itemize}
\end{definition}
A $g$-twisted $V$-module is also called an ordinary module, and an ordinary $g$-twisted module is admissible.

Let $M=\oplus_{\lambda\in\C}M_{\lambda}$ be a $g$-twisted module. The contragredient module $M'$ is defined as follows:
\[
M'=\oplus_{\lambda\in\C}M_{\lambda}^*,
\]
the vertex operator $Y_M'(a, x)$ is defined by 
\begin{equation}\label{eq:contragredient-module}
\< Y_M'(a, x)f, v\> = \< f, Y_{M}(e^{xL_{(1)}}(-x^{-2})^{L_{(0)}}a, x^{-1})v\>
\end{equation}
for $a\in V$, $f\in M'$ and $v\in M$. One can prove (cf. \cite{FHL93, X95}) the following:
\begin{theorem}\label{thm:contragredient}
    $(M', Y_M')$ is a $g^{-1}$-twisted module and $(M'', Y_M'')=(M, Y_M)$. $M$ is irreducible if and only if $M'$ is irreducible. 
\end{theorem}

Let $M$ be a weak $g$-twisted $V$-module. For $a\in V$, $v\in M$, define 
\begin{equation}\label{eq:Y-circ}
    Y_{M}^{\circ}(a, x)v=Y_{M}(e^{xL_{(1)}}(-x^{-2})^{L_{(0)}}a, x^{-1})v\in M((x^{-\frac{1}{T}})).
\end{equation}
In the proof of \cref{thm:contragredient} (cf. \cite{FHL93, X95}), one actually proves the following opposite Jacobi identity:
\begin{equation}\label{eq:opposite-jacobi}
    \begin{aligned}
        \MoveEqLeft
        x_0^{-1}\delta(\frac{x_1-x_2}{x_0})Y_{M}^{\circ}(b, x_2)Y_{M}^{\circ}(a, x_1)v-x_0^{-1}\delta(\frac{-x_2+x_1}{x_0})Y_{M}^{\circ}(a, x_1)Y_{M}^{\circ}(b, x_2)v\\
        &=x_1^{-1}\delta(\frac{x_2+x_0}{x_1})(\frac{x_2+x_0}{x_1})^{-\frac{r}{T}}Y_{M}^{\circ}(Y(a, x_0)b, x_2)v,
    \end{aligned}
\end{equation}
where $a\in V^{r}$, $b\in V$ and $v\in M$.

The following definition is a generalization of the so-called right weak $V$-module defined in \cite{Li02}:
\begin{definition}
    A right weak $g$-twisted $V$-module is a vector space $M$ equipped with a linear map $Y_M(-, x)$ from $V$ to $(\End M)[[x^{-\frac{1}{T}}, x^{\frac{1}{T}}]]$ such that:
    \begin{itemize}
        \item For $a\in V^r, v\in M$, $Y_M(a, x)v=\sum_{n\in\frac{1}{T}\Z}a_{(n)}vx^{-n-1}\in x^{\frac{r}{T}}M((x^{-1}))$;
        \item $Y_M(\vac, x)=id_M$;
        \item \textbf{Opposite Jacobi identity}: For $a\in V^r, b \in V$,
        \[x_0^{-1}\delta(\frac{x_1-x_2}{x_0})Y_M(b, x_2)Y_M(a, x_1)-x_0^{-1}\delta(\frac{-x_2+x_1}{x_0})Y_M(a, x_1)Y_M(b, x_2)\]
        \[=x_1^{-1}\delta(\frac{x_2+x_0}{x_1})(\frac{x_2+x_0}{x_1})^{-\frac{r}{T}}Y_{M}(Y(a, x_0)b, x_2).\]
    \end{itemize}
\end{definition}
A right $g$-twisted $V$-module is defined to be a right weak $g$-twisted module $W$ carrying a $\C$-grading given by the $L_{(0)}$-spectrum, and satisfying the two restrictions in \cref{def:ordinary-module}. Similar to the untwisted case dealt with in \cite{Li02}, we remark that for a left module, $L_{(n)}$ acts locally nilpotently for $n\geq 1$, while for a right module, it is $L_{(-n)}$ that acts locally nilpotently for $n\geq 1$.

For a right weak $g$-twisted module $(W, Y_{W})$, we have the following $L_{(-1)}$-derivative property:  
\begin{equation}
    [L_{(-1)}, Y_{W}(a, x)]=-\odv{x}Y_{W}(a, x) \quad \text{for\ any\ } a\in V.
\end{equation}
Hence, the $L_{(-1)}$-conjugation formula for a right module $W$ is 
\begin{equation}\label{eq: L(-1)-conjugation-right}
    Y_{W}(a, x+z_0)=e^{-z_0L_{(-1)}}Y_{W}(a, x)e^{z_0L_{(-1)}}.
\end{equation}
Another identity we will be using later is the $L_{(1)}$-conjugation formula for twisted modules. Let $(M, Y_{M})$ be a weak $g$-twisted module, the proof of \cite[Lemma 5.2.3]{FHL93} also works for twisted module $M$ if $L_{(1)}$ is locally nilpotent on $M$. Thus, for any $a\in V$ and $z_0\in \C$, we have
\begin{equation}
    e^{-z_0L_{(1)}}Y_{M}(a, x)e^{z_0L_{(1)}}=Y_{M}\left(e^{-z_0(1+z_0x)L_{(1)}}(1+z_0x)^{-2L_{(0)}}a, \frac{x}{1+z_0x}\right).
\end{equation}
We also need the following identity (cf. \cite{FHL93}): 
\begin{equation}\label{eq:conjugation-formula-*}
    x_{1}^{-L_{(0)}}e^{xL_{(1)}}x_{1}^{L_{(0)}}a=e^{xx_{1}L_{(1)}}a \quad \text{for}\ a\in V.
\end{equation}

Using the proof of \cite[Theorem 5.2.1, Proposition 5.3.1]{FHL93}, one can prove the following result:
\begin{proposition}\label{prop:left-right-module-correspondence}
    Let $U$ be a vector space and $Y_{U}(-, x)$ is a linear map from $V$ to $U[[x^{-\frac{1}{T}}, x^{\frac{1}{T}}]]$. Then $(U, Y_{U})$ is a left weak $g$-twisted $V$-module if and only if $(U, Y_{U}^{\circ})$, defined by \labelcref{eq:Y-circ}, is a right weak $g$-twisted $V$-module.
\end{proposition}

Let $(W, Y_{W})$ be a right weak $g$-twisted module. Following \cite{Li02}, for any $a\in V$ and $z_0\in \C$, define 
\begin{equation}
    Y^{(z_0)}_{W}(a, x)=Y_{W}(a, x+z_0)=e^{z_0\odv{x}}Y_{W}(a, x)\in W((x^{-\frac{1}{T}})).
\end{equation}
The following proposition is obtained in \cite{LS23}. It is a twisted analog of \cite[Proposition 2.9]{Li02}.
\begin{proposition}\label{prop:right-module-translation}
    Let $(W, Y_{W})$ be a right weak $g$-twisted module and $z_0\in \C$. Then $(W, Y^{(z_0)}_{W})$ is also a right weak $g$-twisted module. Furthermore, if $L_{(-1)}$ acts locally nilpotently on $W$, then $e^{-z_0L_{(-1)}}$ is an isomorphism from $(W, Y_{W})$ to $(W, Y_{W}^{Z_0})$.
\end{proposition}

Since the $L_{(1)}$-conjugation formula holds for weak $g$-twisted modules, combining \cref{prop:left-right-module-correspondence} and \cref{prop:right-module-translation}, and using an argument similar to \cite[Remark 2.10]{Li02},  we have the following result, which will be used a lot later:
\begin{proposition}\label{prop:L(-1)-resulted-module}
    Let $(W, Y_{W})$ be a left weak $g$-twisted module and $z_0$ be a complex number. Define
    \begin{equation}
    Y_{W}^{[z_0]}(a, x)=Y_{W}\left(e^{-z_0(1+z_0x)L_{(1)}}(1+z_0x)^{-2L_{(0)}}a, \frac{x}{1+z_0x}\right)
\end{equation}
    for any $a\in V$.
    Then $(W, Y_{W}^{[z_0]})$ is also a left weak $g$-twisted module.
    Furthermore, if $L_{(1)}$ acts locally nilpotently on $W$, then $e^{-z_0L_{(1)}}$ is an isomorphism from $(W, Y_{W})$ to $(W, Y_{W}^{[z_0]})$.
\end{proposition}
\begin{proof}
    The same calculation as in \cite[Remark 2.10]{Li02} shows that for any $a\in V$ and $w\in W$,
    \begin{equation}
        Y_{W}^{[z_0]}(a, x)w=((Y_{W}^{\circ})^{(z_0)})^{\circ}(a, x)w.
    \end{equation}
    The first part is a consequence of \cref{prop:left-right-module-correspondence} and \cref{prop:right-module-translation}.
    Suppose $L_{(1)}$ acts locally nilpotently on $W$, then by the $L_{(1)}$-conjugation formula, we have
    \begin{equation}
        ((Y_{W}^{\circ})^{(z_0)})^{\circ}(a, x)=e^{-z_0L_{(1)}}Y_{W}(a, x)e^{z_0L_{(1)}}.
    \end{equation}
    This proves the second part. 
\end{proof}

For any $s\in \Z$, let $\of{s}$ be the residue of $s$ modulo $T$. For any $x\in \frac{1}{T}\Z$, let $\Tilde{x}:=\of{Tx}$. Then $0\leq \Tilde{x}<T$ and for any $x\in \frac{1}{T}\Z$, we have $x=\lf x \rf+\frac{\Tilde{x}}{T}$.
\begin{remark}
    Suppose $a$ is a homogeneous element in $V^r$, then 
    \begin{equation}
        \Res_{x}x^{m}Y_{W}^{[z_0]}(a, x)=0 \quad \text{unless}\ \Tilde{m}=r.
    \end{equation}
\end{remark}

For $0\leq i, r<T$ and $x\in \frac{1}{T}\Z$, let 
    $\de_{i}(r)=1$ if $i\geq r$; otherwise $\de_{i}(r)=0$.
Set
\begin{equation}\label{eq:lambda-number}
    \lambda(x, r)=-1+\lf x \rf +\de_{\Tilde{x}}(r)+\frac{r}{T}.
\end{equation}
Now we recall the $A_{g, n}(V)$ theory ($n\in \frac{1}{T}\N$) developed in \cite{DLM98c}. Suppose $n\in \frac{1}{T}\N$. Let $O_{g, n}(V)$ be the linear span of all $a\circ_{g, n} b$ and $L_{(-1)}a+L_{(0)}a$, where 
\begin{equation}
    a\circ_{g, n} b=\Res_x\frac{(1+x)^{\wt a+\lambda(n, r)}}{x^{2\lf n \rf+\de_{\Tilde{n}}(r)+\delta_{\Tilde{n}}(T-r)+1}}Y(a, x)b
\end{equation}
for $a\in V^r$ homogeneous and $b\in V$. Define a product $\ast_{g, n}$ on $V$ by
\begin{equation}
    a\ast_{g, n}b=\sum_{i=0}^{\lf n \rf}(-1)^i\binom{\lf n \rf+i}{i}\Res_x\frac{(1+x)^{\wt a+\lf n \rf}}{x^{\lf n \rf+i+1}}Y(a, x)b
\end{equation}
for $a\in V^0$ and 
\begin{equation}
    a\ast_{g, n}b=0
\end{equation}
for $a\in V^r$ with $r\neq 0$.

For $n\in \frac{1}{T}\N$ and a weak $g$-twisted module $M$, let 
\begin{equation}
    \Omega_n(M)=\{u\in M \mid a_{(\wt a-1+k)}u=0 \text{ for } a\in V, k\in \frac{1}{T}\N, \text{ and } k>n\}.
\end{equation}
Or  equivalently, let $\Omega_{n}(M)$ be the space consisting of vectors $u\in M$ such that
\begin{equation}
    x^{\wt a+\lambda(n, r)}Y_{M}(a, x)\in M[[x]]
\end{equation}
for homogeneous $a\in V^{r}$, where $0\leq r<T$. Denote $\Res_{x}x^{\wt a-1}Y_{M}(a, x)$ by $o^{M}(a)$.
\begin{proposition}\cite{DLM98c}\label{prop:o-vanishing}
    Let $M$ be a weak $g$-twisted $V$-module. Then for $a, b\in V$ and $v\in \Omega_n(M)$, 
    \begin{equation}
        o^{M}(a\ast_{g, n}b)v=o^{M}(a)o^{M}(b)v, 
    \end{equation}
and 
    \begin{equation}
        o^{M}(a)v=0 \text{ for } a\in O_{g, n}(V).    
    \end{equation}
\end{proposition}

Set $A_{g, n}(V)=V/O_{g, n}(V)$, denote the image of $a\in V$ in $A(V)$ by $[a]$, then we have the following results \cite{DLM98c}:
\begin{theorem}\label{thm:twistedZhu}
    Let $(V, Y, \vac, \omega)$ be a vertex operator algebra and $g$ a finite automorphism of $V$ such that $g^T=1$. Let $M=\bigoplus_{i\in \frac{1}{T}\N}M(i)$ be an admissible $g$-twisted module of $V$. Then for $n\in \frac{1}{T}\N$,
    \begin{enumerate}
        \item $A_{g, n}(V)$ is an associative algebra with respect to the product $\ast_{g, n}$. $[\vac]$ is the unit and $[\omega]$ lies in its center. The identity map on $V$ induces an algebra epimorphism from $A_{g, n}(V)$ to $A_{g, n-\frac{1}{T}}(V)$;
        \item $\Omega_n(M)$ is an $A_{g, n}(V)$-module with $[a]$ acting as $o^{M}(a)$ for homogeneous $a$;
        \item The linear map $\theta$ from $V$ to $V$ defined by $\theta(a)=e^{L(1)}(-1)^{L(0)}a$ for $a\in V$ induces an anti-isomorphism from $A_{g, n}(V)$ to $A_{g^{-1}, n}(V)$.
    \end{enumerate}
\end{theorem}

The following result will be used later, see \cite[Lemma 2.6]{Li22}.
\begin{proposition}\label{prop:Omega-is-submodule}
    Let $W$ be a weak $g$-twisted module. Set $\Omega_{n}(W)=0$ for $n<0$. Then, for any homogeneous $a\in V$ and $m, n\in \frac{1}{T}\N$, 
    \begin{equation}
        a_{(m)}\Omega_{n}(W)\subseteq \Omega_{n+\wt a_{(m)}}(W),
    \end{equation}
    where $\wt a_{(m)}=\wt a-m-1$. Furthermore, the space 
    \begin{equation}
        \Omega_{\infty}(W)=\cup_{n\in \frac{1}{T}\N} \Omega_{n}(W)
    \end{equation}
    is a weak submodule of $W$.
\end{proposition}
\begin{proof}
    The proof for the untwisted case is given in \cite[Lemma 3.8]{Li01a}, which can be extended to the twisted case with a minor modification. We omit it.
\end{proof}

By a similar argument as in \cite[Lemma 3.8]{Li01a}, for a weak $g$-twisted module $W$ and any $n\in \frac{1}{T}\N$, $L_{(1)}$ acts locally nilpotently on $\Omega_{n}(W)$. Hence, $e^{z_0L_{(1)}}$ is well defined on $\Omega_{n}(W)$ for any $z_0\in \C$. The following results are generalizations of \cite[Propositions 3.12, 3.13]{Li01a} to weak twisted modules. Since the proof therein also works here, we omit it.
\begin{proposition}\label{prop:deformed-A_{g, n}(V)-module}
    Let $(W, Y_{W})$ be a $g$-twisted module and $z_0$ be a complex number. Then $\Omega_{n}(W, Y_{W})=\Omega_{n}(W, Y_{W}^{[z_0]})$ for any $n\in \frac{1}{T}\N$, and $e^{-z_0L_{(1)}}$ is an isomorphism of $A_{g, n}(V)$-module from $\Omega_{n}(W, Y_{W})$ to $\Omega_{n}(W, Y_{W}^{[z_0]})$. Furthermore, for homogeneous $a\in V$ and $m\in \frac{1}{T}\Z$, 
    \begin{equation}\label{eq:deformed-residue}
        \Res_{x}x^{m}Y_{W}^{[z_0]}(a, x)=\Res_{x}x^{m}(1-z_0x)^{2\wt a-m-2}Y_{W}(e^{-z_0(1-z_0x)^{-1}L_{(1)}}a, x).
    \end{equation}
\end{proposition}

Next, we recall the twisted regular representation theory developed in \cite{LS23}. From now on, we let $g_1, g_2$ be two commuting automorphisms of $V$, $T$ a positive integer such that $g_1
^T=g_2^T=1$. Set $g = (g_1g_2)^{-1}$. Then $V$ has the following eigenspace decomposition:
\begin{equation}
    V=\oplus_{0\leq j_1, j_2<T}V^{(j_1, j_2)},
\end{equation}
where 
\begin{equation}\label{eq:common-eigenvector-space}
    V^{(j_1, j_2)}=\{v\in V\mid g_kv=e^{\frac{2\pi \iu j_k}{T}}v, k=1, 2\}.
\end{equation}

Let $W$ be a weak $g$-twisted module. Denote by $W^{*}$ the full dual space of $W$, and define a linear map 
\[Y_{W}^{*}\colon V \longrightarrow \left(\End(W^*)\right)[[x^{-\frac{1}{T}}, x^{\frac{1}{T}}]]\]
by the equation
\begin{equation}\label{eq:Y*-action}
    \< Y_{W}^*(a, x)f, u\>=\< f, Y_{W}(e^{xL_{(1)}}(-x^{-2})^{L_{(0)}}a, x^{-1})u\>
\end{equation}
for $a\in V$, $f\in W^*$, and $u\in W$. 

We recall the following definitions in \cite{LS23}:
\begin{definition}
    Let $(W, Y_{W})$ be a weak $g$-twisted module, and let $z$ be a nonzero complex number. Denote by $\mathfrak{D}^{(z)}_{g_1, g_2}(W)$ the set of all linear functionals $\alpha$ in $W^*$ satisfying the condition that for every $a\in V^{(j_1, j_2)}$ with $0\leq j_1, j_2<T$, there exists $k\in \N$ such that
    \begin{equation}\label{eq: P(z)-linear functionals-1}
        x^{\frac{j_2}{T}}(x-z)^{k+\frac{j_1}{T}}Y_{W}^*(a, x)\alpha\in W^*((x)),
    \end{equation}
    i.e., for any $a\in V^{(j_1, j_2)}$, there exist $l, k\in \N$ such that 
    \begin{equation}\label{eq: P(z)-linear functionals-2}
        x^{l+\frac{j_2}{T}}(x-z)^{k+\frac{j_1}{T}}Y_{W}^*(a, x)\alpha\in W^*[[x]].
    \end{equation}
\end{definition}
Let $\alpha\in \mathfrak{D}_{g_1, g_2}^{(z)}(W)$. For $a\in V^{(j_1, j_2)}$ with $0\leq j_1, j_2<T$, define 
\begin{equation}\label{eq:Y-R}
    Y^{R}_{g}(a, x)\alpha=e^{-(k+\frac{j_1}{T})\pi \iu}(z-x)^{-k-\frac{j_1}{T}}\left[(x-z)^{k+\frac{j_1}{T}}Y_{W}^*(a, x)\alpha\right],
    \end{equation}
where $k$ is any nonnegative integer such that \labelcref{eq: P(z)-linear functionals-1} holds. 

\begin{remark}\label{rmk:arg>=pi}
    Define $(Y^{R}_{g}\circ g_1)(a, x)=Y^{R}_{g}(g_1a, x)$ for $a\in V$.  
    Note that 
    \begin{equation}
        (-z+x)^{-k-\frac{j_1}{T}}=
        \begin{cases}
        e^{-(k+\frac{j_1}{T})\pi \iu}(z-x)^{-k-\frac{j_1}{T}}& \text{when arg}(z)<\pi,\\
        e^{\frac{2\pi \iu j_1}{T}}e^{-(k+\frac{j_1}{T})\pi \iu}(z-x)^{-k-\frac{j_1}{T}}& \text{when arg}(z)\geq \pi.
        \end{cases}
    \end{equation}
    For $a\in V^{(j_1, j_2)}$, since 
    \begin{equation}
        Y^{R}_{g}(g_1a, x)=e^{\frac{2\pi \iu j_1}{T}}Y^{R}_{g}(a, x),
    \end{equation}
    we have
    \begin{equation}
        (x-z)^{k+\frac{j_1}{T}}Y_{W}^*(a, x)\alpha=
        \begin{cases}
        (-z+x)^{k+\frac{j_1}{T}}Y^{R}_{g}(a, x)& \text{when arg}(z)<\pi,\\
        (-z+x)^{k+\frac{j_1}{T}}(Y^{R}_{g}\circ g_1)(a, x)& \text{when arg}(z)\geq \pi.
        \end{cases}
    \end{equation}
\end{remark}

Note that for $\alpha\in \mathfrak{D}^{(z)}_{g_1, g_2}(W)$ and $a\in V^{(j_1, j_2)}$, it was proved in \cite{LS23} that \labelcref{eq: P(z)-linear functionals-2} implies that there exist $l, k\in \N$(the same $l, k$ as in \labelcref{eq: P(z)-linear functionals-2}) such that 
\begin{equation}\label{eq:P(z)-linear functionals-3}
    (x+z)^{l+\frac{j_2}{T}}x^{k+\frac{j_1}{T}}Y_{W}^*(a, x+z)\alpha\in W^*[[x]],
\end{equation}
i.e., there exists $l\in \N$ such that 
\begin{equation}
    (x+z)^{l+\frac{j_2}{T}}x^{\frac{j_1}{T}}Y_{W}^*(a, x+z)\alpha\in W^*((x)).
\end{equation}
Now define
\begin{equation}\label{eq:Y-L}
    Y_{g}^{L}(a, x)\alpha=(z+x)^{-l-\frac{j_2}{T}}\left[(x+z)^{l+\frac{j_2}{T}}Y_{W}^*(a, x+z)\alpha\right].
\end{equation}
One of the main results in \cite{LS23} is the following:
\begin{theorem}\label{thm:twisted-regular-representation}
    Let $(W, Y_{W})$ be a weak $g$-twisted module and $z$ be a nonzero complex number. Then:
    \begin{enumerate}[(1)]
        \item The pair $\left(\mathfrak{D}_{g_1, g_2}^{(z)}(W), Y_{g}^R\right)$ carries a weak $g_2$-twisted $V$-module structure;
        \item The pair $\left(\mathfrak{D}_{g_1, g_2}^{(z)}(W), Y_{g}^L\right)$ carries a weak $g_1$-twisted $V$-module structure;
        \item For $a, b\in V$, we have $Y_{g}^L(a, x_1)Y_{g}^{R}(b, x_2)=Y_{g}^R(b, x_2)Y_{g}^{L}(a, x_1)$;
        \item $\mathfrak{D}_{g_1, g_2}^{(z)}(W)$ is a weak $g_1\otimes g_2$-twisted $V\otimes V$-module with the vertex operator map $Y_{g}^{\reg}(\cdot, x)$ defined by
        \begin{equation}
            Y_{g}^{\reg}(a\otimes b, x)=Y_{g}^{L}(a, x)Y_{g}^{R}(b, x) \quad \text{for}\ a, b\in V.
        \end{equation}
    \end{enumerate}
\end{theorem}

It is also proved in \cite{LS23} that $Y_{W}^{*}$, $Y_{g}^{R}$ and $Y_{g}^{L}$ satisfies the following relation:
\begin{proposition}\label{prop:R-L-*}
    Let $\alpha\in \mathfrak{D}_{g_1, g_2}^{(z)}(W)$, $a\in V^{(j_1, j_2)}$ with $j_1, j_2\in Z$. Then
    \begin{equation}\label{eq:R-L-*}
    \begin{aligned}
        \MoveEqLeft
        x_0^{-1}\delta(\frac{x_1-z}{x_0})(\frac{x_1-z}{x_0})^{\frac{j_1}{T}}Y_{W}^{*}(a,x_1)\alpha\\
        -&e^{\frac{j_1}{T}\pi \iu}x_0^{-1}\delta(\frac{-z+x_1}{x_0})(\frac{z-x_1}{x_0})^{\frac{j_1}{T}}Y_{g}^{R}(a,x_1)\alpha\\
        =&x_1^{-1}\delta(\frac{z+x_0}{x_1})(\frac{z+x_0}{x_1})^{\frac{j_2}{T}}Y_{g}^{L}(a, x_0)\alpha. 
    \end{aligned}
    \end{equation}
\end{proposition}
Following \cite[Definition 3.4]{Li22}, we give a similar one:
\begin{definition}\label{def:Omega_{m, n}}
    Let $(M, Y_{M})$ be a weak $g_1\otimes g_2$-twisted $V\otimes V$-module. For $a\in V$, set 
    \begin{equation}
        Y^1(a, x)=Y_{M}(a\otimes \vac, x), \qquad Y^2(a, x)=Y_{M}(\vac\otimes a, x).
    \end{equation}
    Then $(M, Y^1)$ is a weak $g_1$-twisted $V$-module and $(M, Y^2)$ is a weak $g_2$-twisted $V$-module. For $m, n\in \frac{1}{T}\N$, let $\Omega_{m, n}(M)$ be the set of all vectors $v\in M$ such that
    \begin{equation}
        x^{\wt a+\lambda(m, j_1)}Y^1(a, x)v, \quad x^{\wt a+\lambda(n, j_2)}Y^2(a, x) \in M[[x]]
    \end{equation}
    for homogeneous $a$ in $V^{(j_1, j_2)}$, where $0\leq j_1, j_2<T$.
\end{definition}
By definition,
\begin{equation}
    \Omega_{m, n}(M)=\Omega_{m}(M, Y^1)\cap \Omega_{n}(M, Y^2).
\end{equation}
Since the actions of $Y^1$ and $Y^2$ commute, by \cref{thm:twistedZhu}$, \Omega_{m, n}(M)$ is an $A_{g_1, m}(V)$-module and $A_{g_2, n}(V)$-module. Furthermore, it is an $A_{g_1, m}(V)\otimes A_{g_2, n}(V)$-module. 

Let $W$ be a weak $g$-twisted module. For homogeneous $a\in V^{(j_1, j_2)}$, and $v\in W$, define
\begin{equation}
    a\circ_{g_2^{-1}, g_1, m}^{n} v=\Res_{x}\frac{(1+x)^{\lambda(m, j_1)}}{x^{\lambda(m, j_1)+\lambda(n, j_2)+2}}Y_{W}\left((1+x)^{L_{(0)}}a, x\right)v.
\end{equation}
Following \cite{Li22}, let $O^{\dag}_{g_2^{-1}, g_1, n, m}(W)$ be the space spanned by all $a\circ_{g_2^{-1}, g_1, m}^{n} v$ for $a\in V$ and $v\in W$. Furthermore, let 
\begin{equation}
    O'_{g_2^{-1}, g_1, n, m}(W)=O^{\dag}_{g_2^{-1}, g_1, n, m}(W)+(L_{(-1)}+L_{(0)}+m-n)(W).
\end{equation} 
We have the following result similar to \cite[Lemma 2.1.2]{Z96} and \cite[Lemma 3.3]{DJ08b}.
\begin{proposition}\label{prop:k-s-O}
    For homogeneous $a\in V^{(j_1, j_2)}$, and $k, s\in \N$ such that $k\geq s$,
    \begin{equation}
        \Res_{x}\frac{(1+x)^{\wt a+\lambda(m, j_1)+s}}{x^{\lambda(m, j_1)+\lambda(n, j_2)+2+k}}Y_{W}(a, x)v\in O^{\dag}_{g_2^{-1}, g_1, n, m}(W).
    \end{equation}
    In particular, for $p\in \frac{1}{T}\N$, we have
    \begin{equation}
        O^{\dag}_{g_2^{-1}, g_1, n, p}(W)\subseteq O^{\dag}_{g_2^{-1}, g_1, n, m}(W) \quad \text{if}\ p\geq m,
    \end{equation}
    \begin{equation}
        O^{\dag}_{g_2^{-1}, g_1, p, m}(W)\subseteq O^{\dag}_{g_2^{-1}, g_1, n, m}(W) \quad \text{if}\ p\geq n.
    \end{equation}
\end{proposition}
\begin{proof}
    The proof follows from Zhu's argument in \cite{Z96}.
\end{proof}
Set
\begin{equation}
    A^{\dag}_{g_2^{-1}, g_1, n, m}(W)=W/O^{\dag}_{g_2^{-1}, g_1, n, m}(W),
\end{equation}
\begin{equation}
    A^{\diamond}_{g_2^{-1}, g_1, n, m}(W)=W/O^{'}_{g_2^{-1}, g_1, n, m}(W).
\end{equation}

\section{\texorpdfstring{$A^{\dag}_{g_2^{-1}, g_1, n, m}(W)$}{Ag2-1, g1, n, m(W)} and \texorpdfstring{$\Omega_{m, n}(\mathfrak{D}_{g_1, g_2}^{(-1)}(W))$}{m, n, Dg1, g2-1(W)}}
In this section, we study the spaces $A^{\dag}_{g_2^{-1}, g_1, n, m}(W)$ and $\Omega_{m, n}(\mathfrak{D}_{g_1, g_2}^{(-1)}(W))$. By making an identification between $(A^{\dag}_{g_2^{-1}, g_1, n, m}(W))^{*}$ and $\Omega_{m, n}(\mathfrak{D}_{g_1, g_2}^{(-1)}(W))$, we prove $A^{\dag}_{g_2^{-1}, g_1, n, m}(W)$ has an $A_{g_2^{-1}, n}(V)$-$A_{g_1, m}(V)$-bimodule structure, which extends the result of \cite{Li22}.

Note that $\mathfrak{D}_{g_1, g_2}^{(-1)}(W)$ is a weak $g_1\otimes g_2$-twisted $V\otimes V$-module. We give our first main result:
\begin{proposition}\label{prop:A-Omega-D}
    Let $W$ be a weak $g$-twisted module, $f\in W^{*}$ and $m, n\in \frac{1}{T}\N$. Then $f\in (A^{\dag}_{g_2^{-1}, g_1, n, m}(W))^{*}$ if and only if 
    \begin{equation}\label{eq:characterization-of-dual}
        x^{\wt a+\lambda(n, j_2)}(x+1)^{\wt a+\lambda(m, j_1)}Y^{*}_{W}(a, x)f\in W^{*}[[x]]
    \end{equation}
    holds for every homogeneous $a\in V^{(j_1, j_2)}$. Furthermore, we have 
    \begin{equation}\label{eq:A-D}
        (A^{\dag}_{g_2^{-1}, g_1, n, m}(W))^{*}=\Omega_{m, n}(\mathfrak{D}_{g_1, g_2}^{(-1)}(W)).
    \end{equation}
\end{proposition}
\begin{proof}
     Suppose $f\in (A^{\dag}_{g_2^{-1}, g_1, n, m}(W))^{*}$, then $f|_{O^{\dag}_{g_2^{-1}, g_1, n, m}(W)}=0$. Thus for any $k\in \N$, homogeneous $a\in V^{(j_1, j_2)}$, and $u\in W$, by \cref{prop:k-s-O}, we have
     \begin{align*}
         \MoveEqLeft
         \Res_{x}x^{\wt a+\lambda(n, j_2)+k}(x+1)^{\wt a+\lambda(m, j_1)}\<Y^{*}_{W}(a, x)f, u\>\\
         &=\Res_{x}x^{\wt a+\lambda(n, j_2)+k}(x+1)^{\wt a+\lambda(m, j_1)}\<f, Y_{W}(e^{xL_{(1)}}(-x^{-2})^{L_{(0)}}a, x^{-1})u\>\\
         &=(-1)^{\wt a}\Res_{x}x^{\lambda(n, j_2)+k-\wt a}(x+1)^{\wt a+\lambda(m, j_1)}\<f, Y_{W}(e^{xL_{(1)}}a, x^{-1})u\>\\
         &=(-1)^{\wt a}\Res_{x}\frac{(1+x)^{\wt a+\lambda(m, j_1)}}{x^{\lambda(m, j_1)+\lambda(n, j_2)+2+k}}\<f, Y_{W}(e^{x^{-1}L_{(1)}}a, x)u\>\\
         &=(-1)^{\wt a}\sum_{i\geq 0}\Res_{x}\frac{1}{i!}\frac{(1+x)^{\wt (L_{(1)}^{i}a)+\lambda(m, j_1)+i}}{x^{\lambda(m, j_1)+\lambda(n, j_2)+2+i+k}}\<f, Y_{W}(L_{(1)}^{i}a, x)u\>\\
         &=0.
     \end{align*}
     Therefore \labelcref{eq:characterization-of-dual} holds.

     Let $f\in W^{*}$. Suppose \labelcref{eq:characterization-of-dual} holds for for any homogeneous $a\in V^{(j_1, j_2)}$. Then for any homogeneous $a\in V^{(j_1, j_2)}$ and $u\in W$, by \cite[Proposition 5.3.1]{FHL93}, we have
     \begin{align*}
         \MoveEqLeft
         \Res_{x}\frac{(1+x)^{\wt a+\lambda(m, j_1)}}{x^{\lambda(m, j_1)+\lambda(n, j_2)+2}}\<f, Y_{W}(a, x)u\>\\
         &=\Res_{x}\frac{(1+x)^{\wt a+\lambda(m, j_1)}}{x^{\lambda(m, j_1)+\lambda(n, j_2)+2}}\<Y_{W}^{*}(e^{xL_{(1)}}(-x^{-2})^{L_{(0)}}a, x^{-1})f, u\>\\
         &=\Res_{x}(-1)^{\wt a}x^{\wt a+\lambda(n, j_2)}(x+1)^{\wt a+\lambda(m, j_1)}\<Y_{W}^{*}(e^{x^{-1}L_{(1)}}a, x)f, u\>\\
         &=\Res_{x}(-1)^{\wt a}\sum_{i\geq 0}\frac{1}{i!}x^{\wt (L_{(1)}^{i}a)+\lambda(n, j_2)}(x+1)^{\wt (L_{(1)}^{i}a)+\lambda(m, j_1)+i}\\
         &\hspace{1.3cm} \times \<Y_{W}^{*}(L_{(1)}^{i}a, x)f, u\>\\
         &=0.
     \end{align*}
    This implies that $f\in (A^{\dag}_{g_2^{-1}, g_1, n, m}(W))^{*}$, hence proves the first assertion.

    Suppose $f\in (A^{\dag}_{g_2^{-1}, g_1, n, m}(W))^{*}$. By the previous proof, \labelcref{eq:characterization-of-dual} holds for every homogeneous $a\in V^{(j_1, j_2)}$, hence $f\in \mathfrak{D}_{g_1, g_2}^{(-1)}(W)$. By the definition of $Y_{g}^{R}(-, x)$ and \cref{rmk:arg>=pi}, we have 
    \begin{align*}
        \MoveEqLeft
        x^{\wt a+\lambda(n, j_2)}(x+1)^{\wt a+\lambda(m, j_1)}Y^{*}_{W}(a, x)f\\
        &=x^{\wt a+\lambda(n, j_2)}(1+x)^{\wt a+\lambda(m, j_1)}(Y^{R}_{g}\circ g_1)(a, x)f\in \mathfrak{D}_{g_1, g_2}^{(-1)}(W)[[x]],
    \end{align*}
    which implies 
    \begin{equation}\label{eq:Omega_n}
        x^{\wt a+\lambda(n, j_2)}Y^{R}_{g}(a, x)f\in \mathfrak{D}_{g_1, g_2}^{(-1)}(W)[[x]].
    \end{equation}
    On the other hand, \labelcref{eq:characterization-of-dual} implies that
    \begin{equation}
        (x-1)^{\wt a+\lambda(n, j_2)}x^{\wt a+\lambda(m, j_1)}Y^{*}_{W}(a, x-1)f\in W^{*}[[x]]
    \end{equation}
    holds for every homogeneous $a\in V^{(j_1, j_2)}$. By the definition of $Y_{g}^{L}(-, x)$, we have
    \begin{align*}
        \MoveEqLeft
        (x-1)^{\wt a+\lambda(n, j_2)}x^{\wt a+\lambda(m, j_1)}Y^{*}_{W}(a, x-1)f\\
        &=(-1+x)^{\wt a+\lambda(n, j_2)}x^{\wt a+\lambda(m, j_1)}Y^{L}_{g}(a, x)f\in \mathfrak{D}_{g_1, g_2}^{(-1)}(W)[[x]],
    \end{align*}
    which implies
    \begin{equation}\label{eq:Omega_m}
        x^{\wt a+\lambda(m, j_1)}Y^{L}_{g}(a, x)f\in \mathfrak{D}_{g_1, g_2}^{(-1)}(W)[[x]].
    \end{equation}
    Letting $M=W, Y^1=Y_{g}^{L}$ and $Y^2=Y_{g}^{R}$ in \cref{def:Omega_{m, n}}, we see that \labelcref{eq:Omega_n} and \labelcref{eq:Omega_m} means exactly that $f\in \Omega_{m, n}(\mathfrak{D}_{g_1, g_2}^{(-1)}(W))$.

    Suppose $f\in \Omega_{m, n}(\mathfrak{D}_{g_1, g_2}^{(-1)}(W))$. By letting $z=-1$ in \labelcref{eq:R-L-*}, we get
    \begin{align*}
        \MoveEqLeft
        x_0^{-1}\delta(\frac{x_1+1}{x_0})(\frac{x_1+1}{x_0})^{\frac{j_1}{T}}Y_{W}^{*}(a,x_1)\alpha\\
        &-e^{\frac{j_1}{T}2\pi \iu}x_0^{-1}\delta(\frac{1+x_1}{x_0})(\frac{1+x_1}{x_0})^{\frac{j_1}{T}}Y_{g}^{R}(a,x_1)\alpha\\
        &=x_1^{-1}\delta(\frac{-1+x_0}{x_1})(\frac{-1+x_0}{x_1})^{\frac{j_2}{T}}Y_{g}^{L}(a, x_0)\alpha. 
    \end{align*}
    Applying $\Res_{x_0}x_0^{\wt a+\lambda(m, j_1)}x_1^{\wt a+\lambda(n, j_2)}$ to the above equation, we get
    \begin{align*}
        \MoveEqLeft
        x_1^{\wt a+\lambda(n, j_2)}(x_1+1)^{\wt a+\lambda(m, j_1)}Y_{W}^{*}(a, x_1)\alpha\\
        &=e^{\frac{j_1}{T}2\pi \iu}x_1^{\wt a+\lambda(n, j_2)}(1+x_1)^{\wt a+\lambda(m, j_1)}Y_{g}^{R}(a, x_1)\alpha\\
        &+\Res_{x_0}x_0^{\wt a+\lambda(m, j_1)}x_1^{\wt a+\lambda(n, j_2)}x_1^{-1}\delta(\frac{-1+x_0}{x_1})(\frac{-1+x_0}{x_1})^{\frac{j_2}{T}}Y_{g}^{L}(a, x_0)\alpha.
    \end{align*}
    $f\in \Omega_{m, n}(\mathfrak{D}_{g_1, g_2}^{(-1)}(W))$ implies that the right-hand side lies in $W^{*}[[x_1]]$. This completes the proof.
\end{proof}
\cref{thm:twisted-regular-representation} together with \cref{prop:A-Omega-D} implies that $(A^{\dag}_{g_2^{-1}, g_1, n, m}(W))^{*}$ has an $A_{g_1, m}(V)\otimes A_{g_2, n}(V)$-module structure. Next, we will explore the related bimodule structure on $A^{\dag}_{g_2^{-1}, g_1, n, m}(W)$ through the natural pairing between $A^{\dag}_{g_2^{-1}, g_1, n, m}(W)$ and $(A^{\dag}_{g_2^{-1}, g_1, n, m}(W))^{*}$. 

The following result is a consequence of \cref{prop:deformed-A_{g, n}(V)-module}.
\begin{corollary}\label{coro:deformed-action}
    Let $(M, Y_{M})$ be a weak $\sigma$-twisted module for some $\sigma\in \Aut(V)$ satisfying $\sigma^{T}=1$ and $z_0\in \C$. Define a bilinear map from $V\times M$ to $M$ by 
    \begin{equation}
        a\bullet_{(z_0)}u=\Res_{x}x^{\wt a-1}(1-z_0x)^{\wt a-1}Y_{M}(e^{-z_0(1-z_0x)^{-1}L_{(1)}}a, x)u
    \end{equation}
    for $a\in V$ homogeneous and $u\in M$. This bilinear map induces an $A_{\sigma, n}(V)$-module structure on $\Omega_{n}(M)$ for every $n\in \frac{1}{T}\N$. 
\end{corollary}
\begin{proof}
     It is simply because $a\bullet_{(z_0)}u=\Res_{x}x^{\wt a-1}Y_{M}^{[z_0]}(a, x)u$ for homogeneous $a\in V$ and $u\in M$.
\end{proof}

Let $W$ be a weak $g$-twisted module and $z_0\in \C$. We apply \cref{coro:deformed-action} to $\mathfrak{D}_{g_1, g_2}^{(-1)}(W)$, which is a weak $g_1\otimes g_2$-twisted module, hence a weak $g_1$-twisted and $g_2$-twisted $V$-module simultaneously. The $g_1$-twisted action is given by $Y_{g}^{L}(-, x)$ and $g_2$-twisted action is given by $Y_{g}^{R}(-, x)$. For later use, we let
\begin{equation}
    a\bullet_{(z_0)}^{L}f=\Res_{x}x^{\wt a-1}(Y_{g}^{L})^{[z_0]}(a, x)f
\end{equation}
and 
\begin{equation}
    a\bullet_{(z_0)}^{R\circ g_1}f=\Res_{x}x^{\wt a-1}(Y_{g}^{R}\circ g_1)^{[z_0]}(a, x)f
\end{equation}
for $a\in V$ homogeneous and $f\in \mathfrak{D}_{g_1, g_2}^{(-1)}(W)$. Note that $\left(\mathfrak{D}_{g_1, g_2}^{(-1)}(W), Y_{g}^{R}\circ g_1\right)$ is also a weak $g_2$-twisted module and
\begin{equation}
    \Omega_{n}\left(\mathfrak{D}_{g_1, g_2}^{(-1)}(W), Y_{g}^{R}\right)=\Omega_{n}\left(\mathfrak{D}_{g_1, g_2}^{(-1)}(W), Y_{g}^{R}\circ g_1\right)
\end{equation}
for any $n\in \frac{1}{T}\N$.
Thus "$\bullet_{(z_0)}^{L}$" (resp. "$\bullet_{(z_0)}^{R\circ g_1}$") gives an $A_{g_1, m}(V)$-module (resp. $A_{g_2, n}(V)$-module) action on $(A^{\dag}_{g_2^{-1}, g_1, n, m}(W))^{*}$.

\begin{definition}
    Let $W$ be a weak $g$-twisted $V$-module and $m, p, n\in \frac{1}{T}\N$. Define a bilinear map from $W\times V$ to $W$ by 
    \begin{equation}\label{eq:right-action}
        u{\dast}_{g_1, m}^{n}a=\Res_{x}\sum_{i=0}^{\lf m \rf}\binom{\lambda(n, j_2)+i}{i}(-1)^{-\lambda(n, j_2)}\frac{(1+x)^{i-1}}{x^{\lambda(n, j_2)+i+1}}Y_{W}((1+x)^{L_{(0)}}a, x)u,
    \end{equation}
    and a bilinear map from $V\times W$ to $W$ by 
    \begin{equation}\label{eq:general-left-action}
        \begin{aligned}
            &a{\uast}_{g_2^{-1}, m, p}^{n}u=\Res_{x}\sum_{i=0}^{\lf p \rf}\binom{\lambda(m, j_1)+n-p+i}{i}(-1)^{i}\frac{(1+x)^{\lambda(m, j_1)}}{x^{\lambda(m, j_1)+n-p+i+1}}\\
            &\hspace{2cm} \times Y_{W}\left((1+x)^{L_{(0)}}a, x\right)u,
        \end{aligned}
    \end{equation}
    for $a\in V^{(j_1, j_2)}$ and $u\in W$.
    In particular, if $p=n$, denote $a{\uast}_{g_2^{-1}, m, p}^{n}u$ by $a{\uast}_{g_2^{-1}, m}^{n}u$, then
    \begin{equation}\label{eq:left-action}
        a{\uast}_{g_2^{-1}, m}^{n}u=\Res_{x}\sum_{i=0}^{\lf n \rf}\binom{\lambda(m, j_1)+i}{i}(-1)^{i}\frac{(1+x)^{\lambda(m, j_1)}}{x^{\lambda(m, j_1)+i+1}}Y_{W}((1+x)^{L_{(0)}}a, x)u.
    \end{equation}
\end{definition}

\begin{remark}
    Since $W$ is a weak $g$-twisted module, by the definition of twisted modules, we immediately see that $u{\dast}_{g_1, m}^{n}a=0$ if $j_1\neq 0$, and $a{\uast}_{g_2^{-1}, m, p}^{n}u=0$ if $j_2\not\equiv \Tilde{n}-\Tilde{p}\pmod{T}$.  
\end{remark}

\begin{proposition}\label{prop:pairing-L}
    Let $W$ be a weak $g$-twisted $V$-module and $m, n\in \frac{1}{T}\N$. Then, for $a\in V$, $u\in W$, and $f\in (A^{\dag}_{g_2^{-1}, g_1, n, m}(W))^{*}$,   
    \begin{equation}\label{eq:pairing-L}
        \<a\bullet_{(1)}^{L}f, u\>=\<f, u{\dast}_{g_1, m}^{n}a\>.
    \end{equation}
\end{proposition}
\begin{proof}
Without loss of generality, we may assume $a\in V$ is homogeneous. If $a\in V^{(j_1, j_2)}$ with $j_1\neq 0$, then both sides of \labelcref{eq:pairing-L} are 0. Let $a\in V^{(j_1, j_2)}$ with $j_1=0$, $f\in (A^{\dag}_{g_1^{-1}, g_2, n, m}(W))^{*}$, and $u\in W$. For any $r\in \N$, note that by \cref{prop:A-Omega-D} and the facts that $\wt L_{(1)}^r a=\wt a-r$, $\lambda(m, 0)=\lf m \rf$,
\begin{equation}
    x^{\wt a-r+\lf m \rf}Y_{g}^{L}(L_{(1)}^ra, x)f\in \mathfrak{D}_{g_1, g_2}^{(-1)}(W)[[x]].
\end{equation}
Thus, for any $r, i\in \N$ such that $i>\lf m \rf$,
\begin{equation}
    x^{\wt a-r-1+i+r}Y_{g}^{L}(L_{(1)}^ra, x)f\in \mathfrak{D}_{g_1, g_2}^{(-1)}(W)[[x]].
\end{equation}
Note further by \labelcref{eq:Y-L} and \labelcref{eq:characterization-of-dual}, we have
\begin{equation}
    \begin{aligned}
        \MoveEqLeft
        (-1+x)^{\wt a-r+\lambda(n, j_2)}Y_{g}^{L}(L_{(1)}^{r}a, x)f\\
        &=(x-1)^{\wt a-r+\lambda(n, j_2)}Y_{W}^{*}(L_{(1)}^{r}a, x-1)f.
    \end{aligned}
\end{equation}
Applying $\sum_{r=0}^{\infty}\frac{1}{r!}$ to both sides, we get
\begin{equation}
    \begin{aligned}
        \MoveEqLeft
        (-1+x)^{\wt a+\lambda(n, j_2)}Y_{g}^{L}(e^{(-1+x)^{-1}L_{(1)}}a, x)f\\
        &=(x-1)^{\wt a+\lambda(n, j_2)}Y_{W}^{*}(e^{(x-1)^{-1}L_{(1)}}a, x-1)f.
    \end{aligned}
\end{equation}
Therefore, using \labelcref{eq:conjugation-formula-*} and \labelcref{eq:deformed-residue},
\begin{align*}
    \MoveEqLeft
    \Res_{x}x^{\wt a-1}\< (Y_{g}^{L})^{[1]}(a, x)f, u\>\\
    &=\Res_{x}x^{\wt a-1}(1-x)^{\wt a-1}\< Y_{g}^{L}(e^{-(1-x)^{-1}L_{(1)}}a, x)f, u\>\\
    &=\Res_{x}(-1)^{\wt a-1}x^{\wt a-1}(-1+x)^{\wt a-1}\< Y_{g}^{L}(e^{(-1+x)^{-1}L_{(1)}}a, x)f, u\>\\
    &=\Res_{x}(-1)^{-\wt a+1}x^{\wt a-1}(-1+x)^{-1-\lambda(n, j_2)}\\
    &\hspace{1.3cm} \times (-1+x)^{\wt a+\lambda(n, j_2)}\<Y_{g}^{L}(e^{(-1+x)^{-1}L_{(1)}}a, x)f, u\>\\
    &=\Res_{x}\sum_{i=0}^{\infty}\binom{-1-\lambda(n, j_2)}{i}(-1)^{-\wt a-\lambda(n, j_2)-i}x^{\wt a-1+i}\\
    &\hspace{1.3cm} \times (-1+x)^{\wt a+\lambda(n, j_2)}\< Y_{g}^{L}(e^{(-1+x)^{-1}L_{(1)}}a, x)f, u\>\\
    &=\Res_{x}\sum_{i=0}^{\lf m \rf}\binom{-1-\lambda(n, j_2)}{i}(-1)^{-\wt a-\lambda(n, j_2)-i}x^{\wt a-1+i}\\
    &\hspace{1.3cm} \times (-1+x)^{\wt a+\lambda(n, j_2)}\< Y_{g}^{L}(e^{(-1+x)^{-1}L_{(1)}}a, x)f, u\>\\
    &=\Res_{x}\sum_{i=0}^{\lf m \rf}\binom{-1-\lambda(n, j_2)}{i}(-1)^{-\wt a-\lambda(n, j_2)-i}x^{\wt a-1+i}\\
    &\hspace{1.3cm} \times (x-1)^{\wt a+\lambda(n, j_2)}\< Y_{W}^{*}(e^{(x-1)^{-1}L_{(1)}}a, x-1)f, u\>\\
    &=\Res_{x}\sum_{i=0}^{\lf m \rf}\binom{-1-\lambda(n, j_2)}{i}(-1)^{-\wt a-\lambda(n, j_2)-i}\\
    &\hspace{1.3cm} \times x^{\wt a+\lambda(n, j_2)}(x+1)^{\wt a-1+i}\< Y_{W}^{*}(e^{x^{-1}L_{(1)}}a, x)f, u\>\\
    &=\Res_{x}\sum_{i=0}^{\lf m \rf}\binom{-1-\lambda(n, j_2)}{i}(-1)^{-\wt a-\lambda(n, j_2)-i}x^{\wt a+\lambda(n, j_2)}\\
    &\hspace{1.3cm} \times (x+1)^{\wt a-1+i}\< f, Y_{W}(e^{xL_{(1)}}(-x^{2})^{-L_{(0)}}e^{x^{-1}L_{(1)}}a, x^{-1})u\>\\
    &=\Res_{x}\sum_{i=0}^{\lf m \rf}\binom{-1-\lambda(n, j_2)}{i}(-1)^{-\wt a-\lambda(n, j_2)-i}x^{\wt a+\lambda(n, j_2)}\\
    &\hspace{1.3cm} \times (x+1)^{\wt a-1+i}\< f, Y_{W}((-x^{2})^{-L_{(0)}}a, x^{-1})u\>\\
    &=\Res_{x}\sum_{i=0}^{\lf m \rf}\binom{-1-\lambda(n, j_2)}{i}(-1)^{-\lambda(n, j_2)-i}x^{-\wt a+\lambda(n, j_2)}\\
    &\hspace{1.3cm} \times (x+1)^{\wt a-1+i}\< f, Y_{W}(a, x^{-1})u\>\\
    &=\Res_{x}\sum_{i=0}^{\lf m \rf}\binom{\lambda(n, j_2)+i}{i}(-1)^{-\lambda(n, j_2)}\frac{(1+x)^{i-1}}{x^{\lambda(n, j_2)+i+1}}\< f, Y_{W}\left((1+x)^{L_{(0)}}a, x\right)u\>\\
    &=\<f, u{\dast}_{g_1, m}^{n}a\>.
\end{align*}
\end{proof}

\begin{proposition}\label{prop:pairing-R}
    Let $W$ be a weak $g$-twisted $V$-module and $m, p, n\in \frac{1}{T}\N$. Then, for $a\in V$, $u\in W$, and $f\in (A^{\dag}_{g_1^{-1}, g_2, n, m}(W))^{*}$,  
    \begin{equation}\label{eq:general-pairing-R}
        \Res_{x}x^{-1+n-p}\<(Y_{g}^{R}\circ g_1)^{[-1]}(x^{L_{(0)}}a, x)f, u\>=\<f, \theta(a){\uast}_{g_2^{-1}, m, p}^{n}u\>.
    \end{equation}
    In particular, let $p=n$, we get
    \begin{equation}\label{eq:pairing-R}
        \<a\bullet_{(-1)}^{R\circ g_1}f, u\>=\<f, \theta(a){\uast}_{g_2^{-1}, m}^{n}u\>.
    \end{equation}
\end{proposition}
\begin{proof}
Without loss of generality, we may assume $a\in V$ is homogeneous. If $a\in V^{(j_1, j_2)}$ with $j_2\not\equiv \Tilde{n}-\Tilde{p}\pmod{T}$, then both sides of \labelcref{eq:pairing-R} are 0. Let $a\in V^{(j_1, j_2)}$ with $j_2\equiv\Tilde{n}-\Tilde{p}\pmod{T}$, $f\in (A^{\dag}_{g_2^{-1}, g_1, n, m}(W))^{*}$, and $u\in W$. For any $r\in \N$, note that by \cref{prop:A-Omega-D} and the fact that $\wt L_{(1)}^ra=\wt a-r$,
\begin{equation}
    x^{\wt a-r+\lambda(n, j_2)}Y_{g}^{R}(L_{(1)}^ra, x)f\in \mathfrak{D}_{g_1, g_2}^{(-1)}(W)[[x]].
\end{equation}
Thus, for any $r, i\in \N$ such that $i>\lf p \rf$,
\begin{equation}
    x^{\wt a-r-1+n-p+i+r}Y_{g}^{R}(L_{(1)}^ra, x)f\in \mathfrak{D}_{g_1, g_2}^{(-1)}(W)[[x]].
\end{equation}    
Note further by \cref{rmk:arg>=pi} and \labelcref{eq:characterization-of-dual}, we have
\begin{equation}
    \begin{aligned}
        \MoveEqLeft
        (1+x)^{\wt a-r+\lambda(m, j_1)}(Y_{g}^{R}\circ g_1)(L_{(1)}^{r}a, x)f=(x+1)^{\wt a-r+\lambda(m, j_1)}Y_{W}^{*}(L_{(1)}^{r}a, x)f.
    \end{aligned}
\end{equation}
Applying $\sum_{r=0}^{\infty}\frac{1}{r!}$ to both sides, we get
\begin{equation}
    \begin{aligned}
        \MoveEqLeft
        (1+x)^{\wt a+\lambda(m, j_1)}(Y_{g}^{R}\circ g_1)(e^{(1+x)^{-1}L_{(1)}}a, x)f\\
        &=(x+1)^{\wt a+\lambda(m, j_1)}Y_{W}^{*}(e^{(x+1)^{-1}L_{(1)}}a, x)f.
    \end{aligned}
\end{equation}
Therefore, using \labelcref{eq:conjugation-formula-*} and \labelcref{eq:deformed-residue},
    \begin{align*}
        \MoveEqLeft
        \Res_{x}x^{\wt a-1+n-p}\<(Y_{g}^{R}\circ g_1)^{[-1]}(x^{L_{(0)}}a, x)f, u\>\\
        &=\Res_{x}x^{\wt a-1+n-p}(1+x)^{\wt a-1+p-n}\<(Y_{g}^{R}\circ g_1)(e^{(1+x)^{-1}L_{(1)}}a, x)f, u\>\\
        &=\Res_{x}x^{\wt a-1+n-p}(1+x)^{p-n-\lambda(m, j_1)-1}(1+x)^{\wt a+\lambda(m, j_1)}\\
        &\hspace{1.3cm} \times\<(Y_{g}^{R}\circ g_1)(e^{(1+x)^{-1}L_{(1)}}a, x)f, u\>\\
        &=\Res_{x}\sum_{i=0}^{\lf p \rf}\binom{p-n-\lambda(m, j_1)-1}{i}x^{\wt a-1+n-p+i}(1+x)^{\wt a+\lambda(m, j_1)}\\
        &\hspace{1.3cm} \times\<(Y_{g}^{R}\circ g_1)(e^{(1+x)^{-1}L_{(1)}}a, x)f, u\>\\
        &=\Res_{x}\sum_{i=0}^{\lf p \rf}\binom{p-n-\lambda(m, j_1)-1}{i}x^{\wt a-1+n-p+i}(x+1)^{\wt a+\lambda(m, j_1)}\\
        &\hspace{1.3cm} \times\<Y_{W}^{\ast}(e^{(x+1)^{-1}L_{(1)}}a, x)f, u\>\\
        &=\Res_{x}\sum_{i=0}^{\lf p \rf}\binom{p-n-\lambda(m, j_1)-1}{i}x^{\wt a-1+n-p+i}(x+1)^{\wt a+\lambda(m, j_1)}\\
        &\hspace{1.3cm} \times\<f, Y_{W}(e^{xL_{(1)}}(-x^{-2})^{L_{(0)}}e^{(x+1)^{-1}L_{(1)}}a, x^{-1})u\>\\
        &=\Res_{x}\sum_{i=0}^{\lf p \rf}\binom{p-n-\lambda(m, j_1)-1}{i}x^{\wt a+n-p-1+i}(x+1)^{\wt a+\lambda(m, j_1)}\\
        &\hspace{1.3cm} \times\<f, Y_{W}(e^{x(x+1)^{-1}L_{(1)}}(-x^{-2})^{L_{(0)}}a, x^{-1})u\>\\
        &=\Res_{x}\sum_{i=0}^{\lf p \rf}\binom{p-n-\lambda(m, j_1)-1}{i}x^{p-n-\lambda(m, j_1)-1-i}(1+x)^{\wt a+\lambda(m, j_1)}\\
        &\hspace{1.3cm} \times\<f, Y_{W}(e^{(1+x)^{-1}L_{(1)}}(-1)^{L_{(0)}}a, x)u\>\\
        &=\Res_{x}\sum_{i=0}^{\lf p \rf}\sum_{r\geq 0}\binom{\lambda(m, j_1)+n-p+i}{i}(-1)^{i}\frac{(1+x)^{\wt a-r+\lambda(m, j_1)}}{x^{\lambda(m, j_1)+n-p+i+1}}\\
        &\hspace{1.3cm} \times\<f, Y_{W}(\frac{1}{r!}L_{(1)}^{r}(-1)^{L_{(0)}}a, x)u\>\\
        &=\<f, \theta(a){\uast}_{g_2^{-1}, m, p}^{n}u\>.
    \end{align*}
\end{proof}
\begin{remark}
    The reader should compare $a\uast_{g_2^{-1}, m, p}^{n}b$ with $a[p]\uast_{m}^{n}b$ given by Li in \cite{Li22}. When $g_1=g_2=id$, $a\uast_{id, m, p}^{n}b-a[n-p]\uast_{m}^{n}b\in O^{\dag}_{id, id, n, m}(W)$. Our notations are more like Dong and Jiang's notation in \cite{DJ08b} since we will use some results in \cite{DJ08b} later. 
\end{remark}

Combining \cref{prop:pairing-L} and \cref{prop:pairing-R}, we have our second main result:
\begin{theorem}\label{thm:bimodule}
    Let $W$ be a weak $g$-twisted $V$-module and $m, n\in \frac{1}{T}\N$.
    \begin{enumerate}[(1)]
        \item The bilinear map $W\times V\rightarrow W: (u, a)\mapsto u\dast_{g_1, m}^{n}a$ reduces to a bilinear map $A^{\dag}_{g_2^{-1}, g_1, n, m}(W)\times A_{g_1, m}(V)\rightarrow A^{\dag}_{g_2^{-1}, g_1, n, m}(W)$ such that $A^{\dag}_{g_2^{-1}, g_1, n, m}(W)$ is a right $A_{g_1, m}(V)$-module;
        \item The bilinear map $V\times W\rightarrow W: (a, u)\mapsto a\uast_{g_2^{-1}, m}^{n}u$ reduces to a bilinear map $A_{g_2^{-1}, n}(V)\times A^{\dag}_{g_2^{-1}, g_1, n, m}(W) \rightarrow A^{\dag}_{g_2^{-1}, g_1, n, m}(W)$ such that $A^{\dag}_{g_2^{-1}, g_1, n, m}(W)$ is a left $A_{g_2^{-1}, n}(V)$-module;
        \item $A^{\dag}_{g_2^{-1}, g_1, n, m}(W)$ is an $A_{g_2^{-1}, n}(V)$-$A_{g_1, m}(V)$-bimodule.
    \end{enumerate}
\end{theorem}
\begin{proof}
    \begin{enumerate}[(1)]
        \item  Let $a\in O_{g_1, m}(V)$. Then $a\bullet_{(1)}^{L}f=0$ for any $f\in (A^{\dag}_{g_2^{-1}, g_1, n, m}(W))^{*}$ by \cref{prop:A-Omega-D}. Thus for any $f\in (A^{\dag}_{g_2^{-1}, g_1, n, m}(W))^{*}$ and any $u\in W$, $\<f, u{\dast}_{g_1, m}^{n}a\>=\<a\bullet_{(1)}^{L}f, u\>=0$. Hence $u{\dast}_{g_1, m}^{n}a\in O_{g_2^{-1}, g_1, n, m}^{\dag}(W)$. Let $u\in O_{g_2^{-1}, g_1, n, m}^{\dag}(W)$ and $a\in V$. Since $a\bullet_{(1)}^{L}f\in (A^{\dag}_{g_2^{-1}, g_1, n, m}(W))^{*}$ for any $f\in (A^{\dag}_{g_2^{-1}, g_1, n, m}(W))^{*}$, again $\<a\bullet_{(1)}^{L}f, u\>=0$. Thus $u{\dast}_{g_1, m}^{n}a\in O_{g_2^{-1}, g_1, n, m}^{\dag}(W)$. Therefore, the bilinear map does reduce. For $a, b\in V$, $f\in (A^{\dag}_{g_2^{-1}, g_1, n, m}(W))^{*}$ and $u\in W$, we have 
        \begin{equation}
            \begin{aligned}
                \MoveEqLeft
                \<f, u\dast_{g_1, m}^{n}(b\ast_{g_1, m}a)\>=\<(b\ast_{g_1, m}a)\bullet_{(1)}^{L}f, u\>\\
                &=\<b\bullet_{(1)}^{L}(a\bullet_{(1)}^{L}f), u)\>=\<a\bullet_{(1)}^{L}f, u\dast_{g_1, m}^{n}b\>\\
                &=\<f, (u\dast_{g_1, m}^{n}b)\dast_{g_1, m}^{n}a\>.
            \end{aligned}
        \end{equation}
        Thus, the reduced bilinear map gives a right $A_{g_1, m}(V)$-module structure on $A_{g_2^{-1}, g_1, n, m}^{\dag}(W)$.
        \item Note that $\theta^2=1$ and it is an anti-isomorphism from $A_{g_2^{-1}, n}(V)$ to $A_{g_2, n}(V)$. Thus for any $a\in V$, any $f\in (A^{\dag}_{g_2^{-1}, g_1, n, m}(W))^{*}$, and any $u\in W$, we have $\<f, a\uast_{g_2^{-1}, m}^{n}u\>=\<\theta(a)\bullet_{(-1)}^{R\circ g_1}f, u\>$, and $\theta(O_{g_2^{-1}, n}(V))=O_{g_2, n}(V)$. Now, a similar argument to that in (1) finishes the proof. We remark that it gives rise to a left $A_{g_2^{-1}, n}(V)$-module structure since $\theta$ is an anti-isomorphism.
        \item It suffices to show $(a\uast_{g_2^{-1}, m}^{n}u)\dast_{g_1, m}^{n} b=a\uast_{g_2^{-1}, m}^{n}(u\dast_{g_1, m}^{n} b)$ for any $a, b\in V$ and $u\in W$. Suppose $f\in (A^{\dag}_{g_2^{-1}, g_1, n, m}(W))^{*}$. Then
        \begin{equation}
            \<f, (a\uast_{g_2^{-1}, m}^{n}u)\dast_{g_1, m}^{n} b\>=\<\theta(a)\bullet_{(-1)}^{R\circ g_1}(b\bullet_{(1)}^{L}f), u\>,
        \end{equation}
        \begin{equation}
            \<f, a\uast_{g_2^{-1}, m}^{n}(u\dast_{g_1, m}^{n} b)\>=\<b\bullet_{(1)}^{L}(\theta(a)\bullet_{(-1)}^{R\circ g_1}f), u\>.
        \end{equation}
        Since $Y_{g}^{L}$ and $Y_{g}^{R}$ commutes with each other, $(Y_{g}^{L})^{[1]}$ and $(Y_{g}^{R}\circ g_1)^{[-1]}$ also commutes. Thus $\theta(a)\bullet_{(-1)}^{R\circ g_1}(b\bullet_{(1)}^{L}f)=b\bullet_{(1)}^{L}(\theta(a)\bullet_{(-1)}^{R\circ g_1}f)$. This completes the proof.
    \end{enumerate}
\end{proof}
\begin{remark}
    In \cite{DJ08b}, a bilinear map $V\times V\rightarrow V$: $(a, b)\mapsto a\ast_{g_1, m, p}^{n}b$ was defined to provide both left and right actions on their bimodules $A_{g_1, n, m}(V)$. But it is straightforward to see that when $W=V$ and $g_2=g_1^{-1}$,  $a\uast_{g_1, m, p}^{n} b=a\ast_{g_1, m, p}^{n}b$. So the left action in this paper coincides with the left action defined by Dong and Jiang in \cite{DJ08b}. We will show that the two right actions also coincide later.
\end{remark}
We finish this section with the following proposition:
\begin{proposition}
    Let $W$ be a weak $g$-twisted module. Then we have
    \begin{equation}
        a\uast_{g_2^{-1}, m, p}^{n} O^{\dag}_{g_2^{-1}, g_1, p, m}(W)\subset O^{\dag}_{g_2^{-1}, g_1, n, m}(W),
    \end{equation}
    for any $a\in V$ and $m, p, n\in\frac{1}{T}\N$. We set $O^{\dag}_{g_2^{-1}, g_1, n, m}(W)=W$ for $n<0$ (so that $(A^{\dag}_{g_2^{-1}, g_2, n, m}(W))^{*}=0=\Omega_{n, m}(\mathfrak{D}_{g_1, g_2}^{(-1)}(W))$).
\end{proposition}
\begin{proof}
    It is consequences of \cref{prop:Omega-is-submodule}, \cref{prop:A-Omega-D}, \cref{prop:pairing-L} and \cref{prop:pairing-R}.    
\end{proof}

\section{\texorpdfstring{$A^{\diamond}_{g_2^{-1}, g_1, n, m}(W)$}{Ag2-1, g1, n, m(W)} and \texorpdfstring{$\Omega_{m, n}^{\diamond}(\mathfrak{D}_{g_1, g_2}^{(-1)}(W))$}{m, n, Dg1, g2-1(W)}}
In this section, we study the spaces $A^{\diamond}_{g_2^{-1}, g_1, n, m}(W)$ and $\Omega_{m, n}^{\diamond}(\mathfrak{D}_{g_1, g_2}^{(-1)}(W))$ (defined below). For a fixed $m\in \frac{1}{T}\N$, we show the sum \[\Omega_{m, \square}^{\diamond}(\mathfrak{D}_{g_1, g_2}^{(-1)}(W))=\sum_{n\in \frac{1}{T}\N}\Omega_{m, n}^{\diamond}(\mathfrak{D}_{g_1, g_2}^{(-1)}(W))\] is a direct sum, and carries an admissible $g_2$-twisted $V$-module structure. By making identifications between $(A^{\diamond}_{g_2^{-1}, g_1, n, m}(W))^{*}$ and $\Omega_{m, n}^{\diamond}(\mathfrak{D}_{g_1, g_2}^{(-1)}(W))$ for any $n\in \frac{1}{T}\N$, we show the sum $A^{\diamond}_{g_2^{-1}, g_1, \square, m}(W)=\oplus_{n\in \frac{1}{T}\N}A^{\diamond}_{g_2^{-1}, g_1, n, m}(W)$ is an admissible $g_2^{-1}$-twisted $V$-module, and has $\Omega_{m, \square}^{\diamond}(\mathfrak{D}_{g_1, g_2}^{(-1)}(W))$ as its contragredient module. This extends the results of \cite{DJ08a, DJ08b, Li22}. 

Let $M$ be an arbitrary weak twisted module for some automorphism of $V$. For any $a\in V$, $f\in M^{*}$, and $u\in M$, recall \labelcref{eq:Y*-action}
\begin{equation}
    \< Y_{M}^*(a, x)f, u\>=\< f, Y_{M}(e^{xL_{(1)}}(-x^{-2})^{L_{(0)}}a, x^{-1})u\>.
\end{equation}
For the conformal vector $\omega$, we set $Y_{M}^{*}(\omega, x)=\sum_{n\in \N}L_{(n)}^{*}x^{-n-2}$. Then 
\begin{equation}
    \< L_{(n)}^{*}f, u\>=\< f, L_{(-n)}u\>
\end{equation} for any $n\in \N$. Write 
\begin{equation}
    (Y_{g}^{L})^{[1]}(\omega, x)=\sum_{n\in \N}L^{l}_{(n)}x^{-n-2} \quad \text{and}\quad  (Y_{g}^{R}\circ g_1)^{[-1]}(\omega, x)=\sum_{n\in \N}L^{r}_{(n)}x^{-n-2}.
\end{equation}
\begin{proposition}\label{prop:diamond-A-Omega-D}
    Let $W$ be a weak $g$-twisted module and $m, n\in \frac{1}{T}\N$. Then for $f\in \mathfrak{D}_{g_1, g_2}^{(-1)}(W)$, we have
    \begin{equation}
        (L_{(1)}^{*}+L_{(0)}^{*})f=(L^{r}_{(0)}-L^{l}_{(0)})f,
    \end{equation}
    which is equivalent to 
    \begin{equation}
        \<(L^{r}_{(0)}-L^{l}_{(0)})f, u\>=\<f, (L_{(-1)}+L_{(0)})u\>
    \end{equation}
    for any $u\in W$. Furthermore, $f\in (A^{\diamond}_{g_2^{-1}, g_1, n, m}(W))^{*}$ if and only if 
    \begin{equation}
        f\in \Omega_{n, m}(\mathfrak{D}_{g_1, g_2}^{-1}(W))\quad  \text{and}\quad (L^{r}_{(0)}-L^{l}_{(0)})f=(n-m)f.
    \end{equation}
\end{proposition}
\begin{proof}
    Note that $\omega\in V^{(0, 0)}$. So the proof is the same as the proof of \cite[Proposition 4.1]{Li22}. 
\end{proof}
\begin{theorem}
    Let $W$ be a weak $g$-twisted module and $n, m\in \frac{1}{T}\N$. Then the space $A^{\diamond}_{g_2^{-1}, g_1, n, m}(W)$ has an $A_{g_2^{-1}, n}(V)$-$A_{g_1, m}(V)$-bimodule structure as a quotient bimodule of $A^{\dag}_{g_2^{-1}, g_1, n, m}(W)$.
\end{theorem}
\begin{proof}
    The proof is the same as the proof of \cite[Corollary 4.2]{Li22}, we omit it.
\end{proof}
For $m ,n\in\frac{1}{T}\N$, set 
\begin{equation}\label{eq:Omega-diamond-m-n}
    \Omega_{m, n}^{\diamond}(\mathfrak{D}_{g_1, g_2}^{(-1)}(W)=\{f\in \Omega_{m, n}(\mathfrak{D}_{g_1, g_2}^{(-1)}(W)\mid (L^{r}_{(0)}-L^{l}_{(0)})f=(n-m)f\}.
\end{equation}
For any $a\in V$, set
    \begin{equation}
        (Y_{g}^{R}\circ g_1)^{[-1]}(a, x)=\sum_{p\in \frac{1}{T}\N}a^{r}_{(p)}x^{-p-1}.
    \end{equation}
\begin{proposition}\label{prop:Omega-diamond-fiexed-m}
    Let $W$ be a weak $g$-twisted module and $m$ a fixed number in $\frac{1}{T}\N$. Set $\Omega^{\diamond}_{m, n}(\mathfrak{D}_{g_1, g_2}^{(-1)}(W))=0$ for $n<0$. Then, $\Omega^{\diamond}_{m, \square}(\mathfrak{D}_{g_1, g_2}^{(-1)}(W))$ is an admissible $g_2$-twisted submodule of $\left(\mathfrak{D}_{g_1, g_2}^{(-1)}(W), (Y_{g}^{R}\circ g_1)^{[-1]}\right)$. 
\end{proposition}
\begin{proof}
    By \labelcref{eq:Omega-diamond-m-n}, $\Omega_{m, n}^{\diamond}(\mathfrak{D}_{g_1, g_2}^{(-1)}(W)$ is the eigenspace of $L^{r}_{(0)}-L^{l}_{(0)}$ with eigenvalue $n-m$. Thus $\Omega^{\diamond}_{m, \square}(\mathfrak{D}_{g_1, g_2}^{(-1)}(W))=\oplus_{n\in \frac{1}{T}\N}\Omega^{\diamond}_{m, n}(\mathfrak{D}_{g_1, g_2}^{(-1)}(W))$. 
    Note that $\left[L^{r}_{(0)}, a^{r}_{(p)}\right]=(\wt a-p-1)a^{r}_{(p)}$ and $\left[L^{l}_{(0)}, a^{r}_{(p)}\right]=0$ for any $p\in \frac{1}{T}\Z$. Using \cref{prop:Omega-is-submodule}, we get 
    \begin{equation}\label{eq:Omega-admissible}
        a^{r}_{(p)}\Omega^{\diamond}_{m, n}(\mathfrak{D}_{g_1, g_2}^{(-1)}(W))\subset\Omega^{\diamond}_{m, n+\wt a^{r}_{(p)}}(\mathfrak{D}_{g_1, g_2}^{(-1)}(W)).
    \end{equation}
    This completes the proof.
\end{proof}
We have an analogue of \cref{prop:Omega-diamond-fiexed-m}: 
\begin{proposition}\label{prop:Omega-diamond-fiexed-n}
    Let $W$ be a weak $g$-twisted module and $n$ a fixed number in $\frac{1}{T}\N$. Set $\Omega^{\diamond}_{m, n}(\mathfrak{D}_{g_1, g_2}^{(-1)}(W))=0$ for $m<0$. Then, $\Omega^{\diamond}_{\square, n}(\mathfrak{D}_{g_1, g_2}^{(-1)}(W))$ is an admissible $g_1$-twisted submodule of $\left(\mathfrak{D}_{g_1, g_2}^{(-1)}(W), (Y_{g}^{L})^{[1]}\right)$. 
\end{proposition}
\begin{proof}
    The proof is essentially the same as the proof of \cref{prop:Omega-diamond-fiexed-m}, we omit it. 
\end{proof}
By \cref{prop:diamond-A-Omega-D}, $\Omega^{\diamond}_{m, n}(\mathfrak{D}_{g_1, g_2}^{(-1)}(W))=(A^{\diamond}_{g_2^{-1}, g_1, n, m}(W))^{*}$. Thus $\Omega^{\diamond}_{m, \square}(\mathfrak{D}_{g_1, g_2}^{(-1)}(W))$ is the graded dual of $A^{\diamond}_{g_2^{-1}, g_1, \square, m}(W)$. Next, we will show that $A^{\diamond}_{g_2^{-1}, g_1, \square, m}(W)$ is an admissible $g_2^{-1}$-twisted module.
\begin{lemma}\label{lem:well-define-a(p)}
    Let $W$ be a weak $g$-twisted module. Then we have
    \begin{equation}\label{eq:a-*-O-in-O}
        a\uast_{g_2^{-1}, m, p}^{n} O'_{g_2^{-1}, g_1, p, m}(W)\subset O'_{g_2^{-1}, g_1, n, m}(W)
    \end{equation}
    for any $a\in V$ and $m, p, n\in\frac{1}{T}\N$. We set $O'_{g_2^{-1}, g_1, n, m}(W)=W$ for $n<0$ (so that $(A^{\diamond}_{g_2^{-1}, g_2, n, m}(W))^{*}=0=\Omega^{\diamond}_{n, m}(\mathfrak{D}_{g_1, g_2}^{(-1)}(W))$).
\end{lemma}
\begin{proof}
    For any $a\in V$, $u\in O'_{g_2^{-1}, g_1, p, m}(W)$, and $f\in \Omega^{\diamond}_{m, n}(\mathfrak{D}_{g_1, g_2}^{(-1)}(W))$, by \cref{prop:pairing-R}, we have 
    \begin{equation}
        \<\Res_{x}x^{-1+n-p}(Y_{g}^{R}\circ g_1)^{[-1]}\left(x^{L_{(0)}}\theta(a), x\right)f, u\>=\<f, a\uast_{g_2^{-1}, m, p}^{n}u\>.
    \end{equation}
    Using \cref{prop:diamond-A-Omega-D} and \cref{prop:Omega-diamond-fiexed-m}, we see that $a\uast_{g_2^{-1}, m, p}^{n}u\in O'_{g_2^{-1}, g_1, n, m}(W)$. 
\end{proof}
\begin{definition}
    Let $W$ be a weak $g$-twisted module and $m\in \frac{1}{T}\N$ a fixed number. For homogeneous $a\in V$ and $p\in \frac{1}{T}\Z$, define $a_{(p)}$ to be a degree $\wt a-p-1$ operator on $A^{\diamond}_{g_2^{-1}, g_1, \square, m}(W)$ by
    \begin{equation}
        a_{(p)}(v+O'_{g_2^{-1}, g_1, n, m}(W))=a\uast_{g_2^{-1}, m, n}^{n+\wt a-p-1}v+O'_{g_2^{-1}, g_1, n+\wt a-p-1, m}(W)
    \end{equation}
    for $n\in \frac{1}{T}\N$ and $v\in W$.
\end{definition}
By \cref{lem:well-define-a(p)}, $a_{(p)}$ is well-defined, and when $n+\wt a-p-1<0$, $a_{(p)}=0$ on $A^{\diamond}_{g_2^{-1}, g_1, n, m}(W)$. For convenience, by abuse of notation, we still use $a$, $v$ to denote the images of $a$, $v$ in these Zhu algebras and bimodules when there is no confusion.
\begin{proposition}\label{prop:sum-of-A-is-module}
    Let $W$ be a weak $g$-twisted module. Define a linear map 
    \begin{equation*}
        Y^{\diamond}(-, x): V\rightarrow (\End A^{\diamond}_{g_2^{-1}, g_1, \square, m}(W))[[x^{-\frac{1}{T}}, x^{\frac{1}{T}}]]
    \end{equation*} as follows:
    \begin{equation}
        Y^{\diamond}(a, x)=\sum_{p\in \frac{1}{T}\Z}a_{(p)}x^{-p-1}.
    \end{equation}
    Then, for any fixed $m\in \frac{1}{T}\N$, $(A^{\diamond}_{g_2^{-1}, g_1, \square, m}(W), Y^{\diamond})$ is an admissible $g_2^{-1}$-twisted module and has $\left(\Omega^{\diamond}_{m, \diamond}(\mathfrak{D}_{g_1, g_2}^{(-1)}(W)), (Y^{R}_{g}\circ g_1)^{[-1]}\right)$ as its contragredient module.
\end{proposition}
\begin{proof}
    Similar to \cite[Theorem 4.8]{Li22}, it suffices to show 
    \begin{equation}
        \<(Y_{g}^{R})^{[-1]}(a, x)f, v\>=\<f, Y^{\diamond}(e^{xL_{(1)}}(-x^{-2})^{L_{(0)}}a, x^{-1})v\>
    \end{equation}
    for homogeneous $a\in V$, $f\in \Omega^{\diamond}_{m, n}(\mathfrak{D}_{g_1, g_2}^{(-1)}(W))$, and $v\in W$. By \cref{prop:pairing-R},
    \begin{align*}
        \MoveEqLeft
        \<f, Y^{\diamond}(e^{xL_{(1)}}(-x^{-2})^{L_{(0)}}a, x^{-1})v\>\\
        &=\sum_{i=0}^{\infty}\frac{1}{i!}x^{i-2\wt a}\<f, Y^{\diamond}(L_{(1)}^{i}(-1)^{\wt a}a, x^{-1})v\>\\
        &=\sum_{i=0}^{\infty}\sum_{p\in \frac{1}{T}\Z}\frac{1}{i!}x^{i-2\wt a+p+1}\<f, ((L_{(1)}^{i}(-1)^{\wt a}a)_{(p)}v\>\\
        &=\sum_{i=0}^{\infty}\sum_{p\in \frac{1}{T}\Z}\frac{1}{i!}(-1)^{\wt a}x^{i-2\wt a+p+1}\<f, ((L_{(1)}^{i}a)\uast_{g_2^{-1}, m, n}^{n+\wt a-i-p-1}v\>\\
        &=\Res_{y}\sum_{i=0}^{\infty}\sum_{p\in \frac{1}{T}\Z}\frac{1}{i!}(-1)^{\wt a}x^{i-2\wt a+p+1}y^{\wt a-i-p-2}\\
        &\hspace{1.5cm}\<(Y_{g}^{R}\circ g_1)^{[-1]}\left(y^{L_{(0)}}\theta(L_{(1)}^{i}a), y\right)f, v\>\\
        &=\Res_{y}\sum_{i=0}^{\infty}\sum_{p\in \frac{1}{T}\Z}\sum_{j=0}^{\infty}\frac{1}{i!j!}(-1)^{i}x^{i-2\wt a+p+1}y^{2\wt a-p-2-2i-j}\\
        &\hspace{1.5cm} \times\<(Y_{g}^{R}\circ g_1)^{[-1]}\left(L_{(1)}^{j}L_{(1)}^{i}a, y\right)f, v\>\\
        &=\sum_{i=0}^{\infty}\sum_{p\in \frac{1}{T}\Z}\sum_{j=0}^{\infty}\frac{1}{i!j!}(-1)^{i}x^{i-2\wt a+p+1}
        \<\left(L_{(1)}^{j}L_{(1)}^{i}a\right)^{r}_{(2\wt a-p-2-2i-j)}f, v\>\\
        &=\sum_{i=0}^{\infty}\sum_{j=0}^{\infty}\frac{1}{i!j!}(-1)^{i}x^{i+j}
        \<(Y_{g}^{R}\circ g_1)^{[-1]}\left(L_{(1)}^{j}L_{(1)}^{i}a, x\right)f, v\>\\
        &=\<(Y_{g}^{R}\circ g_1)^{[-1]}(a, x)f, v\>.
    \end{align*}
\end{proof}
Note that for any fixed $m\in \frac{1}{T}\N$, $A^{\diamond}_{g_2^{-1}, g_1, \square, m}(W)=\oplus_{n\in \frac{1}{T}\N}A^{\diamond}_{g_2^{-1}, g_1, n, m}(W)$ is a right $A_{g_1, m}(V)$-module. Let $U$ be another left $A_{g_1, m}(V)$-module. We shall show that $A^{\diamond}_{g_2^{-1}, g_1, \square, m}(W)\otimes_{A_{g_1, m}(V)} U$ is an admissible $g_2^{-1}$-twisted module. When $U=\C$, this is just \cref{prop:sum-of-A-is-module}.

First we consider the space $A^{\diamond}_{g_2^{-1}, g_1, \square, m}(W)\otimes_{\C} U$. By \cref{prop:sum-of-A-is-module}, it is obvious an admissible $g_2^{-1}$-twisted module with the module map given by $Y^{\diamond}(-, x)\otimes id$. Since $A^{\diamond}_{g_2^{-1}, g_1, \square, m}(W)\otimes_{A_{g_1, m}(V)} U$ is a quotient of $A^{\diamond}_{g_2^{-1}, g_1, \square, m}(W)\otimes_{\C} U$ modulo the relations 
\begin{equation}
    v\dast_{g_1, m}^{n} a\otimes u-v\otimes (a.u),
\end{equation}
where $v\in A^{\diamond}_{g_2^{-1}, g_1, n, m}(W)$, $a\in A_{g_1, m}(V)$, $u\in U$, and $n\in \frac{1}{T}\N$, it suffices to show $b_{(p)}$ preserves these relations for any homogeneous $b\in V$ and $p\in \frac{1}{T}\Z$, i.e., 
\begin{equation}
    b\uast_{g_2^{-1}, m, n}^{n+\wt b-p-1}(v\dast_{g_1, m}^{n}a)=(b\uast_{g_2^{-1}, m, n}^{n+\wt b-p-1}v)\dast_{g_1, m}^{n+\wt b-p-1}a \quad \text{in } A^{\diamond}_{g_2^{-1}, g_1, n+\wt b-p-1, m}(W).   
\end{equation}

\begin{proposition}\label{prop:two-actions-commuting}
    For $m, n, p\in \frac{1}{T}\N$, $a, b\in V$, and $v\in W$, we have
    \begin{align}
        &b\uast_{g_2^{-1}, m, p}^{n}(v\dast_{g_1, m}^{p}a)-(b\uast_{g_2^{-1}, m, p}^{n}v)\dast_{g_1, m}^{n}a\in O'_{g_2^{-1}, g_1, n, m}(W).
    \end{align}
    In particular, $A^{\diamond}_{g_2^{-1}, g_1, \square, m}(W)\otimes_{A_{g_1, m}(V)} U$ is an admissible $g_2^{-1}$-twisted module with module map given by $Y^{\diamond}(-, x)\otimes id$. 
\end{proposition}
\begin{proof}
    The proof is similar to that of the third part of \cref{thm:bimodule}. We only need to use the fact that $(Y_{g}^{R}\circ g_1)^{[-1]}$ and $(Y_{g}^{L})^{[1]}$ give two commuting actions on $\mathfrak{D}_{g_1, g_2}^{(-1)}(W)$.
\end{proof}

\section{Dong and Jiang's conjecture}
In this section, we confirm the conjecture proposed by Dong and Jiang in \cite{DJ08a, DJ08b}. For any finitely-ordered $g_1\in \Aut(V)$ such that $g_1^{T}=1$ and $n, m\in \frac{1}{T}\N$, in constructing $A_{g_1, n}(V)$-$A_{g_1, m}(V)$-bimodule $A_{g_1, n, m}(V)$, Dong and Jiang defined a subspace $O_{g_1, n, m}(V)$ of $V$ consisting of three parts 
\begin{equation}
    O_{g_1, n, m}(V)=O'_{g_1, n, m}(V)+O''_{g_1, n, m}(V)+O'''_{g_1, n, m}(V),
\end{equation}
such that $A_{g_1, n, m}(V)=V/O_{g_1, n, m}(V)$. Recall that $O''_{g_1, n, m}(V)$ is the linear span of 
\begin{equation}
    d\uast_{g_1, m, p_3}^{n}\left((a\uast_{g_1, p_1, p_2}^{p_3}b)\uast_{g_1, m, p_1}^{p_3}c-a\uast_{g_1, m, p_2}^{p_3}(b\uast_{g_1, m, p_1}^{p_2}c)\right)
\end{equation}
for all $a, b, c, d\in V$ and $m, n, p_1, p_2, p_3\in \frac{1}{T}\N$. Also, recall that 
\begin{equation}
    O'''_{g_1, n, m}(V)=\sum_{p_1, p_2\in \frac{1}{T}\N}(V\uast_{g_1, p_1, p_2}^{n}O'_{g_1, p_2, p_1})\uast_{g_1, m, p_1}^{n}V.
\end{equation}
Dong and Jiang conjectured $O_{g_1, n, m}(V)=O'_{g_1, n, m}(V)$.
A recent result in \cite{HAN25} shows that $O'''_{g_1, n, m}(V)$ is redundant. We shall show that 
\begin{equation}
    O''_{g_1, n, m}(V)+O'''_{g_1, n, m}(V)\subseteq O'_{g_1, n, m}(V).    
\end{equation}
But firstly, let us show that the left (right) action defined in this paper coincides with that defined by Dong and Jiang.   

 Let $W=V$ and $g_1=g_2^{-1}$. It is straightforward to verify that  $O'_{g_1, g_1, n, m}(V)$ is identical to the space $O'_{g_1, n, m}(V)$ defined by Dong and Jiang in \cite{DJ08b}. So $A_{g_1, n, m}(V)$ is a quotient of $A_{g_1, g_1, n, m}^{\diamond}(V)$ for any $n, m\in \frac{1}{T}\N$. We show that the right action "$\dast_{g_1, m}^{n}$" in this paper is the same as the right action "$\ast_{g_1, m}^{n}$" defined by Dong and Jiang. (The two left actions are the same by definition.) Recall that for $a, b\in V$ such that $g_1a=e^{\frac{2\pi \iu r}{T}}a$ and $g_1^{-1}b=e^{\frac{2\pi \iu s}{T}}b$, where $r, s\in \N$ and $0\leq r, s<T$, Dong and Jiang defined 
\begin{equation}
    a\ast_{g_1, m}^{n}b=\Res_{x}\sum_{i=0}^{\lf m \rf }\binom{\lf n \rf +i}{i}(-1)^{i}\Res_{x}\frac{(1+x)^{\lambda(m, r)}}{x^{\lf n \rf +i+1}}Y\left((1+x)^{L_{(0)}}a, x\right)b,
\end{equation}
if $\overline{\Tilde{m}-\Tilde{n}}=r$, and $a\ast_{g_1, m}^{n}b=0$ otherwise.
While the right action in this paper is defined by 
\begin{equation}
    a\dast_{g_1, m}^{n}b=\Res_{x}\sum_{i=0}^{\lf m \rf}\binom{\lambda(n, s)+i}{i}(-1)^{-\lambda(n, s)}\frac{(1+x)^{i-1}}{x^{\lambda(n, s)+i+1}}Y((1+x)^{L_{(0)}}b, x)a.
\end{equation}
The following proposition shows that the two right actions coincide in $A^{\diamond}_{g_1, g_1, n, m}(V)$.
\begin{proposition}\label{prop:two-right-actions-coincide}
    For $a, b\in V$, we have $a\ast_{g_1, m}^{n}b-a\dast_{g_1, m}^{n}b\in O'_{g_1, g_1, n, m}(V)$. Therefore $a\dast_{g_1, m}^{n}b=a\ast_{g_1, m}^{n} b$ in $A^{\diamond}_{g_1, g_1, n, m}(V)$.
\end{proposition}
\begin{proof}
    Without loss of generality, we may assume $a, b$ to be homogeneous such that $g_1a=e^{\frac{2\pi \iu r}{T}}a$ and $g_1^{-1}b=e^{\frac{2\pi \iu s}{T}}b$, where $r, s\in \N$ and $0\leq r, s<T$. The following congruence relation is well-known (cf. \cite{Z96, DLM98a}):
    \begin{equation}
        e^{xL_{(-1)}}Y(a, -x)b\equiv (1+x)^{n-m-\wt a-\wt b}Y(a, \frac{-x}{1+x})b \pmod{(L_{(-1)}+L_{(0)}+m-n)V}.
    \end{equation}
    \begin{enumerate}[(1)]
        \item When $s\neq 0$, $a\dast_{g_1, m}^{n}b=0$, and $a\ast_{g_1, m}^{n}b$ has $e^{\frac{2\pi \iu (r-s)}{T}}$ as eigenvalue with respect to $g_1$. By \cite[Lemma 3.1]{DJ08b}, if $\overline{\Tilde{m}-\Tilde{n}}\neq \overline{r-s}$, $a\ast_{g_1, m}^{n}b\in O'_{g_1, g_1, n, m}(V)$. If $\overline{\Tilde{m}-\Tilde{n}}=\overline{r-s}$, since $s\neq 0$, $\overline{\Tilde{m}-\Tilde{n}}\neq r$, thus $a\ast_{g_1, m}^{n}b=0$. 
        \item When $s=0$, $a\dast_{g_1, m}^{n}b$ has $e^{\frac{2\pi \iu r}{T}}$ as eigenvalue with respect to $g_1$. If $\overline{\Tilde{m}-\Tilde{n}}\neq r$, $a\ast_{g_1, m}^{n}b=0$ by definition and $a\dast_{g_1, m}^{n}b\in O'_{g_1, g_1, n, m}(V)$ by \cite[Lemma 3.1]{DJ08b} again. If $\overline{\Tilde{m}-\Tilde{n}}=r$, then
        \begin{align*}
            \MoveEqLeft
            a\dast_{g_1, m}^{n}b\\
            &=\Res_{x}\sum_{i=0}^{\lf m \rf}\binom{\lf n \rf +i}{i}(-1)^{\lf n \rf}\frac{(1+x)^{\wt b+i-1}}{x^{\lf n \rf +i+1}}Y(b, x)a\\
            &=\Res_{x}\sum_{i=0}^{\lf m \rf}\binom{\lf n \rf +i}{i}(-1)^{\lf n \rf}\frac{(1+x)^{\wt b+i-1}}{x^{\lf n \rf +i+1}}e^{xL_{(-1)}}Y(a, -x)b\\
            &\equiv \Res_{x}\sum_{i=0}^{\lf m \rf}\binom{\lf n \rf +i}{i}(-1)^{\lf n \rf}\frac{(1+x)^{n-m-\wt a+i-1}}{x^{\lf n \rf +i+1}}Y(a, \frac{-x}{1+x})b\\
            &\equiv \Res_{x}\sum_{i=0}^{\lf m \rf}\binom{\lf n \rf +i}{i}(-1)^{i}\frac{(1+x)^{\wt a+m-n+\lf n \rf}}{x^{\lf n \rf +i+1}}Y(a, x)b\\
            &\equiv a\ast_{g_1, m}^{n}b \pmod{(L_{(-1)}+L_{(0)}+m-n)V}.
        \end{align*}
        Since $(L_{(-1)}+L_{(0)}+m-n)V\subset O'_{g_1, g_1, n, m}(V)$, we are done.
    \end{enumerate}
\end{proof}

So far, we have shown that $A_{g_1, n, m}(V)$ is a quotient bimodule of $A^{\diamond}_{g_1, g_1, n, m}(V)$. Recall that $O'_{g_1, g_1, n, m}(V)=O'_{g_1, n, m}(V)$. 
The following two lemmas will be used later:
\begin{lemma}\label{lem:O-*-a-in-O}
    Let $m, p, n\in \frac{1}{T}\N$ and $b\in V$. Then
    \begin{equation}
        O'_{g_1, n, p}(V)\uast_{g_1, m, p}^{n}b\in O'_{g_1, n, m}(V).
    \end{equation}
\end{lemma}
\begin{proof}
    Let $f\in \Omega_{m, n}(\mathfrak{D}_{g_1, g_1^{-1}}^{(-1)}(V))$ and $a\in O'_{g_1, n, p}(V)$. By \cref{prop:pairing-R}, we have
    \begin{equation}
            \<f, a\uast_{g_1, m, p}^{n}b\>
            =\Res_{x}x^{-1+n-p}\<(Y_{g}^{R}\circ g_1)^{[-1]}\left(x^{L_{(0)}}\theta(a), x\right)f, b\>.
    \end{equation}
    By the proof of \cite[Proposition 4.2]{DJ08b}, we know 
    \begin{equation}
        \theta(O'_{g_1, n, p}(V))\subseteq O'_{g_{1}^{-1}, p, n}(V).
    \end{equation}
    Thus $\theta(a)\in O'_{g_{1}^{-1}, p, n}(V)$.
    The proof of \cite[Lemma 5.2]{DJ08b} tells us that 
    \begin{equation}
        \Res_{x}x^{-1+n-p}(Y_{g}^{R}\circ g_1)^{[-1]}\left(x^{L_{(0)}}\theta(a), x\right)f=0.
    \end{equation}
    Hence $a\uast_{g_1, m, p}^{n}b\in O'_{g_1, n, m}(V)$. 
\end{proof}
\begin{lemma}\label{lem:a-*-b-*-c-associativity}
    For all $a, b, c\in V$ and $m, p_1, p_2, p_3\in \frac{1}{T}\N$, we have
    \begin{equation}
        (a\uast_{g_1, p_1, p_2}^{p_3}b)\uast_{g_1, m, p_1}^{p_3}c-a\uast_{g_1, m, p_2}^{p_3}(b\uast_{g_1, m, p_1}^{p_2}c)\in O'_{g_1, p_3, p_1}(V).
    \end{equation}
\end{lemma}
\begin{proof}
    By \cref{prop:pairing-R}, for any $f\in \Omega_{m, p_3}(\mathfrak{D}_{g_1, g_1^{-1}}^{(-1)}(V))$ and homogeneous $a, b, c\in V$, we have
    \begin{equation}
        \begin{aligned}
        \MoveEqLeft
        \<f, (a\uast_{g_1, p_1, p_2}^{p_3}b){\uast}_{g_1, m, p_1}^{p_3}c\>\\
        &=\Res_{x}x^{-1+p_3-p_1}\<(Y_{g}^{R}\circ g_1)^{[-1]}\left(x^{L_{(0)}}(\theta(a\uast_{g_1, p_1, p_2}^{p_3}b)), x\right)f, c\>,
        \end{aligned}
    \end{equation}
    and 
    \begin{equation}
        \begin{aligned}
            &\<f, a\uast_{g_1, m, p_2}^{p_3}(b\uast_{g_1, m, p_1}^{p_2}c)\>\\
            &=\Res_{x}\Res_{y}x^{-1+p_3-p_2}y^{-1+p_2-p_1}\\
            &\hspace{1.5cm} \<(Y_{g}^{R}\circ g_1)^{[-1]}(y^{L_{(0)}}\theta(b), y)(Y_{g}^{R}\circ g_1)^{[-1]}(x^{L_{(0)}}\theta(a), x)f, c\>.
        \end{aligned}
    \end{equation}
    By the proof of \cite[Proposition 4.2]{DJ08b}, we have
    \begin{equation}
        \theta(a\uast_{g_1, p_1, p_2}^{p_3}b)-\theta(b)\uast_{g_1^{-1}, p_3, p_2}^{p_1}\theta(a)\in O'_{g_1^{-1}, p_1, p_3}(V).
    \end{equation}
    Therefore,
    \begin{equation}
        \begin{aligned}
        &\<f, (a\uast_{g_1, p_1, p_2}^{p_3}b){\uast}_{g_1, m, p_1}^{p_3}c\>\\
        &=\Res_{x}x^{-1+p_3-p_1}\<(Y_{g}^{R}\circ g_1)^{[-1]}\left(x^{L_{(0)}}(\theta(b)\uast_{g_1^{-1}, p_3, p_2}^{p_1}\theta(a)), x\right)f, c\>\\
        &=\Res_{x}\Res_{y}x^{-1+p_3-p_2}y^{-1+p_2-p_1}\\
        &\hspace{1.5cm}\<(Y_{g}^{R}\circ g_1)^{[-1]}(y^{L_{(0)}}\theta(b), y)(Y_{g}^{R}\circ g_1)^{[-1]}(x^{L_{(0)}}\theta(a), x)f, c\>\\
        &=\<f, a\uast_{g_1, m, p_2}^{p_3}(b\uast_{g_1, m, p_1}^{p_2}c)\>,
        \end{aligned}
    \end{equation}
    where the second equality follows from \cite[Lemma 5.2]{DJ08b}, and the third equality follows from the proof of \cite[Lemma 5.1]{DJ08b}. 
    Hence, we have proved
    \begin{equation}
        (a\uast_{g_1, p_1, p_2}^{p_3}b)\uast_{g_1, m, p_1}^{p_3}c-a\uast_{g_1, m, p_2}^{p_3}(b\uast_{g_1, m, p_1}^{p_2}c)\in O'_{g_1, p_3, p_1}(V).
    \end{equation}
\end{proof}
\begin{proposition}
    For any $n, m\in \frac{1}{T}\N$ and a finitely-ordered $g_1\in \Aut(V)$ such that $g_1^T=1$, we have 
    \begin{equation}
        O''_{g_1, n, m}(V)+O'''_{g_1, n, m}(V)\subseteq O'_{g_1, n, m}(V).  
    \end{equation}
\end{proposition}
\begin{proof}
    $O''_{g_1, n, m}(V)\subseteq O'_{g_1, n, m}(V)$ is a consequence of \labelcref{eq:a-*-O-in-O} and \cref{lem:a-*-b-*-c-associativity}, and $O'''_{g_1, n, m}(V)\subseteq O'_{g_1, n, m}(V)$ is a consequence of \labelcref{eq:a-*-O-in-O} and \cref{lem:O-*-a-in-O}.  
\end{proof}

\textbf{Acknowledgement}: The author would like to express his gratitude to Yukun Xiao for his careful reading of the manuscript and pointing out many typos and mistakes. 
\section*{Declarations}

No funding was received for conducting this study.
The authors have no competing interests to declare that are relevant to the content of this article.
No data was used for the research described in the article.



%%       \citet{<label>} ==> Jones et al. [21]
%%       \citep{<label>} ==> [21]
%%

%% If you have a bibdatabase file and want bibtex to generate the
 \bibliographystyle{plain} 
 \bibliography{bib}



\end{document}